\documentclass[smallcondensed]{svjour3}       

\smartqed 

\usepackage{algorithm}
\usepackage{algorithmicx}
\usepackage{amsmath}
\usepackage{booktabs}
\usepackage{bm}
\usepackage{dsfont}
\usepackage{graphicx}
\usepackage{float}
\usepackage{amssymb}
\usepackage{mathtools}
\usepackage{enumerate}
\usepackage{mathdots}
\usepackage{mathrsfs}
\usepackage{multirow}
\usepackage[table]{xcolor}
\usepackage{tabularx}
\usepackage{url}
\usepackage{wasysym}

\newcommand{\comp}{\mathsf{c}}

\newcommand{\alphaMax}{\overline{\alpha}}
\newcommand{\Ftail}{\mkern 1mu\overline{\mkern-1mu{F}\mkern+1.25mu}}

\newcommand{\independent}{\protect\mathpalette{\protect\independenT}{\perp}}
\def\independenT#1#2{\mathrel{\rlap{$#1#2$}\mkern2mu{#1#2}}}

\newcommand{\bfone}{\bm{1}}
\newcommand{\bfzero}{\bm{0}}
\newcommand{\bfmu}{\bm\mu}
\newcommand{\bfSigma}{\bm{\Sigma}}

\newcommand{\bfC}{\bm{C}}
\newcommand{\bfI}{\bm{I}}

\newcommand{\bfX}{\bm{X}}
\newcommand{\bfY}{\bm{Y}}

\newcommand{\bfx}{\bm{x}}

\newcommand{\bft}{\bm{t}}

\DeclareMathOperator{\Prob}{{\mathbb P}}

\renewcommand{\ast}{\top}

\DeclareMathOperator{\Exp}{\mathbb E}
\DeclareMathOperator{\Var}{\mathbb{V}\mathrm{ar}}
\DeclareMathOperator{\Q}{\mathbb Q}

\newcommand{\Corr}{{\mathbb{C}\mathrm{orr}}\,}
\newcommand{\Ind}{\mathds{1}}
\newcommand{\RL}{{\mathbb R}}
\newcommand{\Nat}{{\mathbb N}}

\newcommand{\dd}{\mathop{}\mathopen{}\mathrm{d}}
\newcommand{\e}{{\mathrm{e}}}

\newcommand{\oh}{{\mathrm{o}}}
\newcommand{\Oh}{{\mathcal{O}}}

\newcommand{\eqdistr}{\stackrel{{\scriptstyle \mathcal{D}}}{=}}

\newcommand{\LaplaceDist}{\mathcal{L}}
\newcommand{\NormDist}{\mathcal{N}}
\newcommand{\ExpDist}{\mathcal{E}}
\newcommand{\EllDist}{\mathcal{ELL}}
\newcommand{\IGDist}{\mathcal{IG}}

\renewcommand{\epsilon}{\varepsilon}
\newcommand{\for}[1]{\,,\qquad \text{for } #1}

\newtheorem{Theorem}{Theorem}

\newtheorem{Example}{Example}
\newtheorem{Proposition}{Proposition}
\newtheorem{Corollary}{Corollary}
\newtheorem{Definition}{Definition}
\newtheorem{Remark}{Remark}

\newcommand{\thrm}[1] {Theorem~\ref{#1}}

\mathtoolsset{centercolon}

\renewcommand{\hat}{\widehat}
\renewcommand{\phi}{\varphi}
\usepackage{subfig}
 \usepackage{colortbl}
 \definecolor{Gray}{gray}{0.9}

\newcommand{\BigBinom}[2]{\Bigl({{#1}\atop#2}\Bigr)}

\begin{document}

\title{Efficient simulation for dependent rare events}
\subtitle{with applications to extremes}
\author{Lars N{\o}rvang Andersen \and Patrick J.\ Laub \and Leonardo Rojas-Nandayapa}
\institute{Lars N{\o}rvang Andersen \at
          Aarhus University \\
          \email{larsa@math.au.dk} 
       \and
       Patrick J.\ Laub \at
       University of Queensland and Aarhus University \\
       \email{p.laub@[uq.edu.au$|$math.au.dk]}
       \and
       Leonardo Rojas-Nandayapa \at
       University of Liverpool \\
           \email{leorojas@liverpool.ac.uk}
}

\date{Received: date / Accepted: date}

\maketitle

	\begin{abstract}
	We consider the general problem of estimating probabilities which arise as
	a union of dependent events. We propose a flexible series of
	estimators for such probabilities, and describe variance reduction schemes applied
	to the proposed estimators. We derive efficiency results of the estimators in  rare-event settings, in particular those associated with extremes. Finally, we examine the performance of our estimators in a numerical example.

	\end{abstract}

\section{Introduction}

The estimators in this paper apply to quite general problems, so we will first introduce them in the framework of our main example, namely, as estimators relating to rare maxima of dependent random vectors.  
For a random vector $\bfX = (X_1, \dots, X_d)$ with maximum $M = \max_i X_i$, the first problem we consider is estimating 
\[ \alpha(\gamma) = \Prob(M > \gamma) \,. \] 
This problem has many applications in many areas, for example in actuarial science (e.g.\ default probabilities \cite{RP2}), finance (e.g.\ probability of `knock-out' in a barrier option \cite{cont2010encyclopedia}), survival analysis, reliability \cite{rausand2004system} and engineering (e.g.\ failure probability of a series circuit).

We construct estimators for this probability, which are in terms of 
\[ E(\gamma) = \sum_{i=1}^d \Ind\{X_i > \gamma\} \,,\]
the random variable which counts the number of $X_i$ which exceed $\gamma$.\footnote{We use $\Ind\{ \cdot \}$ to denote the indicator function, and $\Ind\{ \emptyset \} = 1$.} Our two main estimators in this setting are 
\begin{align} 
  \hat{\alpha}_1 \label{alpha_1}
  &= \sum_{i=1}^d \Prob(X_i > \gamma) + \frac1R \sum_{r=1}^R (1 - E_r(\gamma)) \Ind\{E_r(\gamma) \ge 2 \} \,, \text{ and} \\
  \hat{\alpha}_2  \label{alpha_2}
  &= \sum_{i=1}^d \Prob(X_i > \gamma) - \sum_{i=1}^{d-1} \sum_{j=i+1}^d \Prob(X_i > \gamma, X_j > \gamma) \\
  &\quad + \frac1R \sum_{r=1}^R \Big[1 - E_r(\gamma) + \frac{E_r(\gamma) (E_r(\gamma) -1)}{2} \Big] \Ind\{E_r(\gamma) \ge 3\} \,. \notag
\end{align}
where $R \in \Nat$ and the $E_r(\gamma)$s are derived from i.i.d.\ samples of $\bfX$.
The fact that these are unbiased estimators of $\alpha(\gamma)$ follows from
Proposition \ref{prop:means_of_binomial} below. Estimation of 
$\Prob(M > \gamma)$ is a difficult problem and treatments in the literature make distributional assumptions on $\bfX$. 
 One such example is Adler et al.\ \cite{adler2012efficient} where $\bfX$ is assumed to be multivariate normal. In this case, our estimator $\hat{\alpha}_1$, with appropriate importance sampling, is the same as one of the estimators from \cite{adler2012efficient}.
 
The next problem we consider is estimating
\[ \beta_n(\gamma) := \Exp[ Y \Ind\{ E(\gamma) \ge n\} ] \]
for $n=1,\dots,d$ and some random variable $Y$.
We do not make any assumptions of independence between the $\{X_i > \gamma\}$ events themselves or between the events and $Y$.

The subcase of $Y=1$ a.s.\ has some interesting examples:
\[ \beta_1(\gamma) = \Prob(M > \gamma) = \alpha(\gamma) \,, \quad \text{and} \quad 
 \beta_n(\gamma) = \Prob(X_{(n)} > \gamma) \]
where $X_{(1)} \ge X_{(2)} \ge \dots \ge X_{(d)}$ are the order statistics of $\bfX$. The probability of a parallel circuit failing is a simple application for $\Prob(X_{(n)} > \gamma)$. 

Our main $\beta_1$ estimator uses the fact that
\begin{align} \label{extreme-first-partition}
\{M > \gamma\}  := \bigcup_{i=1}^d \{X_i > \gamma\} = \bigcup_{i=1}^d \{X_1 \le \gamma, \dots, X_{i-1} \le \gamma, X_i > \gamma\}
\end{align}
where the events in the union on the right are disjoint.
This supplies a form of $\beta_1$ which is amenable to efficient Monte Carlo estimation:
\begin{equation} \label{extreme-beta_1}
 \beta_1=\sum_{i=1}^d \Exp [ Y \, \Ind\{X_1 \le \gamma, \dots, X_{i-1} \le \gamma \} \mid X_i > \gamma ] \Prob\left(X_i > \gamma \right) \,.
\end{equation}

As previously mentioned, while they are main example and motivation, the extremes considered so far are a very specific instance of estimators. We now
turn our attention to the general set-up treated in the paper. 

Let $A(\gamma) = \cup_{i=1}^d A_i(\gamma)$ be the union of events $A_1(\gamma),\dots,A_d(\gamma)$ for an index parameter $\gamma \in \mathbb{R}$. 
    We consider the problem of estimating $\Prob ( A(\gamma) ) $ when the events are rare, that is, $\Prob(A(\gamma)) \to 0$ as $\gamma \to \infty$.
Define
\[ \alpha(\gamma) := \Prob ( A(\gamma) ) \quad \text{and} \quad \quad E(\gamma) := \sum_{i=1}^d \Ind\{A_i(\gamma)\} \, .
\]
Note that we recover our introductory example by having 
$A_i(\gamma) = \{X_i > \gamma\}$. Aside from this example, $A(\gamma)$ is quite general (a union of arbitrary events)
and many interesting events arising in applied probability and statistics can be
formulated as a union. The quantity $\beta_n(\gamma)$ is reminiscent of \emph{expected shortfall} from risk management \cite{mcneil2015quantitative}.

Traditional Monte Carlo methods are unreliable in the rare-event setting. We will use standard techniques from the \emph{rare-event simulation methodology}, such as importance sampling for variance reduction and applicable 
measures of efficiency: \emph{bounded relative error} and \emph{logarithmic efficiency}, cf.\ \cite{asmussen2007stochastic,glasserman2003monte,rubinstein2011simulation}. The resulting estimators are among the most efficient possible under the
most general assumptions.

The paper is structured as follows. In Sections \ref{scn:alpha_ests} and \ref{scn:beta_ests} we formally introduce our estimators for $\alpha(\gamma)$ and $\beta_n(\gamma)$ respectively, 
we prove their validity, and
show how to combine them with some existing variance reduction techniques; the efficiency properties
for the general estimators are analysed in Section~\ref{scn:efficiency_results}, in addition we further investigate
the efficiency for certain important dependence structures. Finally, we evaluate the numerical performance of the 
estimators in Section~\ref{scn:numerical_results}. 

\section{Estimators of $\alpha$} \label{scn:alpha_ests}

In the following, we first explain the construction of our estimators of $\alpha$, then discuss possible variance reduction schemes. As the $\gamma$ notation can be cumbersome, we simply write $A=A(\gamma)$, $A_i = A_i(\gamma)$, $E = E(\gamma)$, $\alpha = \alpha(\gamma)$ and $\beta_n = \beta_n(\gamma)$. Similarly, we often write $\sum_i$, $\sum_{i<j}$, $\cup_i$, $\cap_i$ for $\sum_{i=1}^d$,$\sum_{1=i<j}^d$, $\cup_{i=1}^d$ and $\cap_{i=1}^d$.

\subsection{Proposed estimators of $\alpha$}

The inclusion--exclusion formula (IEF) provides a representation of $\alpha$ 
as a summation whose terms are decreasing in size. 
The formula is
\begin{equation} \label{incl_excl}
 \alpha = \Prob(A)
  =\sum_{i=1}^d  (-1)^{i+1} \sum_{|I|= i}  \Prob\Big(\bigcap_{i\in I}A_i\Big)  \,.
\end{equation}
The IEF can rarely be used as its summands are increasingly difficult to calculate numerically. 
The $\Prob(A_i)$ terms are typically known, and the $\Prob(A_i,A_j)$ terms can frequently be calculated, however the remaining higher-dimensional terms are normally intractable for numerical integration algorithms (cf.\ the \emph{curse of dimensionality} \cite[Chapter IX]{asmussen2007stochastic}). 
Truncating the summation leads to bias, and indeed by the Bonferroni inequalities we have:
\begin{align}
  \alpha &\le \sum_{i=1}^k  (-1)^{i-1} \sum_{|I|= i} \Prob\Big(\bigcap_{i\in I}A_i\Big) \quad \text{ if } 1\le k < d \text{ and } k \text{ is odd}, \label{bonf_odd} \\
  \alpha &\ge \sum_{i=1}^k  (-1)^{i-1} \sum_{|I|= i} \Prob\Big(\bigcap_{i\in I}A_i\Big) \quad \text{ if } 1<k<d \text{ and } k \text{ is even}. \label{bonf_even}
\end{align}
This higher-order intractability motivates our estimators which use the IEF rewritten in terms of $E = \sum_i \Ind\{ A_i \}$.

\begin{Proposition} \label{prop:means_of_binomial}
For $i = 1,\dots, d$,
\begin{equation} \label{summands} 
  \sum_{|I|=i} \Ind\left\{\cap_{i\in I}A_i\right\} = \BigBinom{E}{i} \Ind\{E \ge i\}  \,.
\end{equation}
\end{Proposition}
\begin{proof}
\begin{align*}
\sum_{|I|=i} \Ind\{ \cap_{i\in I}A_i \} =
\sum_{k=i}^{d} \sum_{|I|=i} \Ind\{ \cap_{i\in I}A_i, E = k \} = \sum_{k=i}^{d} \BigBinom{k}{i} \Ind\{E = k \} = \BigBinom{E}{i} \Ind\{E \ge i\} \,.
\end{align*} 
\qed 
\end{proof}
Taking the expectation of \eqref{summands} gives
\begin{equation*} 
 \sum_{|I|=i} \Prob\Big(\bigcap_{i\in I}A_i\Big) = \Exp\Big[ \BigBinom{E}{i} \Ind\{E \ge i\} \Big] \quad \text{ for } i=1,\dots,d \,.
\end{equation*}
So the following has mean $\alpha$, and forms the nucleus of our $\hat{\alpha}_i$ estimators:
\begin{equation} \label{estimator_incl_excl}
	\sum_{i=1}^d (-1)^{i-1} \BigBinom{E}{i} \Ind\{E\ge i\} \,.
\end{equation}

We present estimators which deterministically \emph{calculate} the first larger terms of the IEF \eqref{incl_excl} and Monte Carlo (MC) \emph{estimate} the remaining smaller terms using sample means of \eqref{summands}.
We begin by constructing the single-replicate estimator $\hat{\alpha}_1$
where the first summand is calculated and the remaining terms are estimated:
\begin{align*}
	\hat{\alpha}_1 :&= \sum_i \Prob(A_i)  + \sum_{i=2}^d \Big[ (-1)^{i-1} 
	 \BigBinom{E}{i} \Ind\{E \ge i\} \Big] \\
	&= \sum_i \Prob(A_i) 
+ (1 - E) \Ind\{E \ge 2 \} \,, \quad \text{using} \quad \sum_{k=0}^n (-1)^{k-1} \BigBinom{n}{k} = 0 \,.
\end{align*}
In identical fashion, the single-replicate estimator calculating the first two terms from the IEF is
\begin{align*}
	\hat{\alpha}_2 
	:=& \sum_i \Prob(A_i) - \sum_{i<j}\Prob(A_i, A_j) + \sum_{i=3}^d \Big[ (-1)^{i-1} \BigBinom{E}{i} \Ind\{E \ge i\} \Big] \\
	=& \sum_i \Prob(A_i) - \sum_{i<j}\Prob(A_i, A_j) + \Big[1 - E + \frac{E (E-1)}{2} \Big] \Ind\{E \ge 3\} \,.
\end{align*}
Thus, for $n \in \{1,\dots,d-1\}$,\footnote{Note that by the IEF, we have
$\hat{\alpha}_d:=\alpha$, so this possibility is ignored.}
\begin{equation} \label{hat_alpha_n}
	\hat{\alpha}_n :=
\sum_{i=1}^n (-1)^{i-1} \sum_{|I|=i} \Prob\Big( \bigcap_{i\in I} A_i \Big)
+ \Big[ \sum_{i=0}^n (-1)^{i} \BigBinom{E}{i} \Big] \Ind\{ E \ge n+1\} \,. 
\end{equation}

Thus, $\{\hat{\alpha}_1,\dots,\hat{\alpha}_{d-1}\}$ is a collection of estimators which allows the user to control the computational division of labour between numerical integration and Monte Carlo estimation. We will furthermore let
$\hat{\alpha}_0$ be the crude Monte Carlo estimator $\Ind\{E \ge 1\}$, and note that this falls under
the definition in $\eqref{hat_alpha_n}$ if we interpret the empty sum as zero.

The $\hat{\alpha}_n$ estimators are of decreasing variance in $n$, however each estimator carries the assumption that one can perform accurate numerical integration for $1$ up to $n$ dimensions. 
As numerical integration can be slow and unreliable in high dimensions we focus on $\hat{\alpha}_1$, and also show the numerical performance of $\hat{\alpha}_2$.

In practice, theses estimators will exhibit very modest improvements
when compared against their truncated IEF counterparts (i.e., the right side of \eqref{bonf_odd} and \eqref{bonf_even}).
When combined with importance sampling, as in Section~\ref{scn:IS}, the
improvement is marked. 
Furthermore, we will show that these estimators
possess desirable efficiency properties which are preserved after combining
with importance sampling.

\subsection{Discussion of $\hat{\alpha}_1$ estimator}

The estimator $\hat{\alpha}_1$ has some nice interpretations. Recall the Boole--Fr\'echet inequalities
\begin{equation} \label{boole-frechet}
	\max_i \, \Prob(A_i) \le \alpha = \Prob( A ) \le \sum_i \Prob(A_i) =: \alphaMax \,.
\end{equation}
The stochastic part of $\hat{\alpha}_1$ is an unbiased estimate of $\alphaMax - \alpha \le 0$. 
That is to say, $\hat{\alpha}_1$ MC estimates the difference between the target quantity
$\alpha$ and its upper bound given by the Boole--Fr\'echet inequalities, $\alphaMax$.
Similarly, we often have
\[ \alpha(\gamma) \sim \sum_i \Prob(A_i(\gamma)) \,,\footnote{Using the standard notation that $f(x) \sim g(x)$ means $\lim_{x \to \infty} f(x)/g(x) = 1$.}
 \]
for example when the $A_i$ exhibit a weak dependence structure.
In this case, we can say that $\hat{\alpha}_1$ MC estimates the difference between $\alpha$ and its (first-order) asymptotic expansion.

\subsection{Relation of $\hat{\alpha}_n$ estimators to control variates}

An alternative construction of $\{\hat{\alpha}_1,\dots,\hat{\alpha}_{d-1}\}$
is to add \emph{control variates} to the crude Monte Carlo estimator
$\hat{\alpha}_0$.
We begin by adding the control variate $E$ to $\hat{\alpha}_0$ with weight $\tau \in \RL$:
\[ \hat{\alpha}_1^{\tau} := \Ind\{ E \ge 1 \} - \tau \big[ E - \sum_i \Prob(A_i) \big] \,. \]
Setting $\tau=1$ means this estimator simplifies to $\hat{\alpha}_1$.
Next, we add the control variates $E$ and ${-}\frac12 E(E-1)$ to $\hat{\alpha}_0$, and setting the corresponding weights to 1 gives $\hat{\alpha}_2$. This pattern goes on.

\subsection{Combining $\hat{\alpha}_1$ with importance sampling} \label{scn:IS}

The family of estimators $\hat{\alpha}_n$ can be combined with the
variance reduction technique called \emph{importance sampling} (IS), cf.\ \cite{asmussen2007stochastic,glasserman2003monte}. Standard IS theory suggests that we should focus on IS distributions where the event of interest $A = \cup_i A_i = \{E \ge 1\}$ occurs almost surely.
A convenient way of constructing such a distribution is as a \emph{mixture distribution}. Say that we condition on $A_i$ with probability
\[ p_i := \frac{\Prob(A_i)}{ \sum_j \Prob(A_j) } = \frac{\Prob(A_i)}{ \alphaMax} \for i=1,\dots,d\,. \]
A heuristic motivation for this selection comes from a rare-event setting where the asymptotic relationship
$\Prob(A_i(\gamma), A_j(\gamma))=\oh(\Prob(A_i(\gamma)))$ often occurs for all $i\neq j$. 
In such a case 
\[ 
 \Prob\left(A_i(\gamma) \mid A(\gamma) \right) 
   = \dfrac{\Prob(A_i(\gamma))}{\sum_j \Prob(A_j(\gamma))(1 + \oh(1))} \,   \sim p_i(\gamma) \,, \quad \text{ as } \gamma \to \infty\,. \]
Now consider the measure
\[ \Q^{[1]}( \mathscr{A} ) = \sum_i p_i \Prob( \mathscr{A} \mid A_i ) \qquad \forall \mathscr{A} \in \mathcal{F} \,,  \]   
which induces the likelihood ratio of $L^{[1]} := \dd \Q^{[1]} / \dd \Prob = \alphaMax/E$. 
As
\[
  \alphaMax + (1-E)\Ind\{E \ge 2\} L^{[1]}
  = \alphaMax \Big( 1 + \frac{1-E}{E} \Big) \\
  = \frac{\alphaMax}{E} \quad \text{ under } \Q^{[1]}\, ,
\]
we can see that $\hat{\alpha}_1$ under this change of measure, with $R\in \Nat$ replicates, is
\begin{equation} \label{alpha_1_IS1}
  \hat{\alpha}_{1}^{[1]} := \frac1R \sum_{r=1}^R \frac{\alphaMax}{E_r^{[1]}} \,, 
\end{equation}
where the superscript ``$[1]$'' indicates that the $E_r^{[1]}$ are (independently) sampled under $\Q^{[1]}$. This estimator corresponds to one from the paper of Adler et al.\ \cite{adler1990introduction}, though applied in a  more general way (they consider rare maxima of normally distributed vectors). 

Continuing in the same pattern, consider the \emph{second-order} IS distributions where $\{E \geq 2\}$ occurs almost surely, to be applied to $\hat{\alpha}_2$.
Say that we choose to condition on $A_i \cap A_j$ with probability
\[ p_{ij} := \frac{\Prob(A_i,A_j)}{ \sum_{m<n} \Prob(A_m,A_n) } = \frac{\Prob(A_i,A_j)}{ q } \for 1 \le i < j \le d\,, \]
defining ${q}:= \sum_{i<j} \Prob(A_i, A_j)$. Now consider the measure
\[ \Q^{[2]}( \mathscr{A} ) = \sum_{i<j} p_{ij} \Prob( \mathscr{A} \mid A_i, A_j)  \qquad \forall \mathscr{A} \in \mathcal{F} \,, \]
which induces a likelihood ratio of
\[ L^{[2]} := \frac{\dd \Q^{[2]} }{ \dd \Prob } = \frac{ {q} }{ \sum_{i<j} \Ind\{A_iA_j\}} 
  = \frac{ {q} }{ \binom{ E }{2} } = \frac{2 {q} }{ E (E - 1) }  \,. \]
Thus, after simplifying, the estimator $\hat{\alpha}_2$ under $\Q^{[2]}$ is
\begin{align} \label{alpha_2_IS_2}
  \hat{\alpha}_{2}^{[2]} := 
  \alphaMax -  \frac{2q}{R} \sum_{r=1}^R \frac{1}{E_r^{[2]}} .
\end{align}


\begin{Remark}
As the $\Q^{[2]}$-mean of $\frac2{E}$ is less than 1, this fraction can be seen as a correction term for the two-term truncation of \eqref{incl_excl}. We know from \eqref{bonf_even} that $\alpha \ge \alphaMax - q$. 

\hfill $\Diamond$
\end{Remark}

Both of these IS algorithms have some extra requirements for their use. The first-order estimators require 
that we can simulate from $\Prob(\,\cdot \mid A_i)$ and can calculate the $\Prob(A_i)$.
The second-order estimator requires that we 
can simulate from $\Prob(\,\cdot~|~A_i, A_j)$ and that we can calculate the $\Prob(A_i)$ and $\Prob(A_i, A_j)$.
In the rare maxima case, integration routines in \textsc{Mathematica} or \textsc{Matlab} can usually calculate these probabilities; it is simulating from the conditional distributions which can be the prohibitive requirement, particularly for $\hat{\alpha}_{2}^{[2]}$. 

\section{Estimators of $\beta_n$} \label{scn:beta_ests}

Now, we turn our attention to the estimation of 
\[ \beta_n := \Exp[ Y \Ind\{ E \ge n\} ] \,. \]
We start with $\beta_1$, and rewrite the partition \eqref{extreme-first-partition} in terms of the general $A_i$:
\begin{align} \label{first-partition}
A := \bigcup_{i=1}^d A_i = A_1 \cup (A_1^\comp A_2) \cup \dots \cup (A_1^\comp \dots A_{d-1}^\comp A_d ) \,.
\end{align}
This gives us (the generalised version of \eqref{extreme-beta_1})
\begin{align*}
  \beta_1 &= \Exp[Y \mid A_1] \Prob(A_1) + \Exp[Y \Ind\{A_1\} \mid A_2] \Prob(A_2) \\
  & \qquad + \dots + \Exp[Y \Ind\{A_1^\comp \dots A_{d-1}^\comp\} \mid A_d] \Prob(A_d) \,.
\end{align*}
If we assume it is possible to sample from the $\Prob(\,\cdot \mid A_i)$ conditional distributions---the same assumption required to use the first-order IS estimator $\hat{\alpha}_{1}^{[1]}$ from Section~\ref{scn:IS}---then each of these conditional expectations can be estimated by sample means:
\begin{align} \label{est-beta_1}
	\hat{\beta}_1 := \sum_{i=1}^d \frac{\Prob(A_i)}{\lceil R/d \rceil} \sum_{r=1}^{\lceil R/d \rceil} 
	Y_{i,r} \Ind\{A_1^\comp \dots A_{i-1}^\comp\}_{i,r} \,.
\end{align}
Here, the $Y_{i,r}$ and $\Ind\{\cdot\}_{i,r}$ are sampled independently and conditional on $A_i$.
The following proposition gives the partition of the event $\{E \ge i\}$:

\begin{Proposition}\label{Prop.Partition}
 Consider a finite collection of events $\{A_1,\dots,A_d\}$ and for each subset 
 $I\subset\{1,2,\dots,d\}$ define~\footnote{Using the convention that $\cap_\emptyset=\Omega$.}
   \begin{equation*}
    B_I:=\bigcap_{j\in I} A_j,\qquad
    C_I:=\bigcap_{\substack{k\notin I,\\k<\max I}} A_k^c.
 \end{equation*}
 Then
 \begin{equation}\label{eq.partition}
  \{E\ge m\}=\bigcup_{|I|=m} B_I = \bigcup_{|I|=m} B_I C_I.
 \end{equation}
 Moreover, the collection of sets $\{B_IC_I:|I|=m\}$ is disjoint.
\end{Proposition}

\begin{proof}
 
 The first equality in \eqref{eq.partition} is straightforward from the definition of
 the random variable $E$. For the second equality
 we note that the relation $\boldsymbol{\supseteq}$ follows trivially;
 to prove the opposite relation $\boldsymbol{\subseteq}$ it remains to show that
 if  $\omega$ is such that $\omega\in B_I$ and $\omega\notin C_I$, 
 then there exists $I'$ such that $|I'|=m$ and $\omega\in B_{I'}C_{I'}$.  Notice that
 if $\omega\notin C_I$, then there exists a nonempty set $J$ satisfying 
 $\max J<\max I$, with $j\in J$ if and only if $\omega\notin A_j^c$. 
 Select $I'$ as the set formed
 by the smaller $m$ elements of $I\cup J$.  In consequence,
 \begin{equation*}
  \omega\in
  \Big(\bigcap_{j\in I\cup J}A_j\Big) \Big(\bigcap_{\substack{k\notin I\cup J,\\k\le\max I}} A_k^c\Big)
  \subseteq
  \Big(\bigcap_{j\in I'}A_j\Big) \Big(\bigcap_{\substack{k\notin I',\\k\le\max I'}} A_k^c\Big)
  =B_{I'}C_{I'}.
 \end{equation*}  
 This completes the proof of the second equivalence in \eqref{eq.partition}.
 
 Next we show that the collection of sets $\{B_IC_I:|I|=m\}$ is disjoint.  Consider two sets of indexes
 $I_1$ and $I_2$ such that
 $|I_1|=|I_2|=m$ and $I_1\neq I_2$. Take $i$ such that
 $i\in I_1$, $i\notin I_2$ and w.l.o.g.\ further assume that $i<\max I_2$.  
 Then $B_{I_1}\subseteq A_i$ while $C_{I_2}\subseteq A_i^c$. \qed
\end{proof}
This proposition implies that
 \begin{align*}
  \beta_n&=\Exp\Big[Y\Ind\Big\{\bigcup_{|I|=n} B_I\Big\}\Big]
    =\Exp\Big[Y\Ind\Big\{\bigcup_{|I|=n} B_IC_I\Big\}\Big] 
  =\sum_{|I|=n}\Exp\left[Y\Ind\left\{C_I\right\}  
     \big|B_I\right]\Prob\left(B_I\right).
 \end{align*}  
  Therefore, if (i)
 reliable estimates of  $\Prob\big(B_I\big)$ are available, and
 (ii) it is possible to simulate from the conditional measures
 $\Prob\left(\, \cdot \mid B_I\right)$,
then the following is an unbiased estimator of $\Exp[Y\Ind\{E\ge n\}]$:
 \begin{equation}\label{beta_hat}
 \hat{\beta}_n
   :=\sum_{|I|=n} \frac{ \Prob(B_I) }{\lceil R/\binom{d}{n} \rceil} 
   \sum_{r=1}^{\lceil R/\binom{d}{n} \rceil}
   Y_{I,r} \Ind\{ C_I \}_{I,r} \,.
 \end{equation}
 Here, similar to before, $Y_{I,r}$ and $\Ind\{ \cdot \}_{I,r}$ denote independent sampling conditioned on $B_I$.
 
Notice that a permutation of the sets $A_1,\dots,A_d$ will result in
a different collection of events $C_I$, and also a slightly different estimator.  

\subsection{Applying $\hat{\beta}_i$ to estimate $\alpha$}

The $\hat{\beta}_i$ estimators can be used in various ways to estimate the probability
$\alpha=\Prob(A)$. The simplest way is to set $Y=1$ a.s.\ in $\hat{\beta}_1$ \eqref{beta_hat},
leading to the estimator 
\begin{align} \label{beta_on_alpha_1}
\widehat{(\beta_1 \ddagger \alpha)} := \Prob(A_1) +  
\sum_{i=2}^d \frac{\Prob(A_i)}{\lceil R/(d{-}1) \rceil} \hspace{-5pt} \sum_{r=1}^{\lceil R/(d{-}1) \rceil} 
	\Ind\{A_1^\comp \dots A_{i-1}^\comp\}_{i,r}  \,,
\end{align}
using the notation from \eqref{est-beta_1}. Note, we achieve minor improvement in \eqref{beta_on_alpha_1} over \eqref{beta_hat} when $Y=1$ a.s.\ as $\Exp[1 \mid A_1] = 1$ does not require estimation.

More effective estimators can be constructed if we use $\hat{\beta}_n$ to estimate terms from $\hat{\alpha}_{n-1}$ \eqref{hat_alpha_n}.
We label the random terms in $\hat{\alpha}_n$ as
\begin{equation}\label{residual}
 R_n := \Big[ \sum_{i=0}^n (-1)^{i} \BigBinom{E}{i} \Big] \Ind\{ E \ge n+1\}, \quad \text{and say} \quad \mathcal{R}_n :=\Exp[R_n] \,.
\end{equation}
Now, if we choose $Y:=\sum_{i=0}^{n-1} (-1)^{i} \binom{E}{i}$ then it is obvious that
\begin{equation*}
 \beta_n:=\Exp\Big\{ \Big[ \sum_{i=0}^{n-1} (-1)^{i} \BigBinom{E}{i} \Big] \Ind\{ E \ge n\}\Big\}=\mathcal{R}_{n-1}.
\end{equation*}
This leads to the set of estimators
\begin{align*}
	\widehat{(\beta_n \ddagger \alpha)} 
	&:= \sum_{i=1}^{n-1} (-1)^{i-1} \sum_{|I|=i} \Prob\Big( \bigcap_{i\in I} A_i \Big) \\
	&\quad +
	\sum_{|I|=n} \frac{ \Prob(B_I) }{\lceil R/\binom{d}{n} \rceil} 
   \sum_{r=1}^{\lceil R/\binom{d}{n} \rceil}
   \Big[ \sum_{i=0}^{n-1} (-1)^{i} \BigBinom{E_{I,r}}{i} \Big]
	\Ind\{ E \ge n \}_{I,r} \,, 
\end{align*}
for $n = 2, \dots d-1$.
In particular, for $n=2$
\begin{align} \label{beta_on_alpha_2}
	\widehat{(\beta_2 \ddagger \alpha)} 
	&:= \sum_i \Prob(A_i) + 
	\sum_{i<j} \frac{ \Prob(A_i, A_j) }{\lceil R/\binom{d}{2} \rceil} 
   \sum_{r=1}^{\lceil R/\binom{d}{2} \rceil}
   (1 - E_{ij,r}) \Ind\{ E \ge 2 \}_{ij,r} \,, 
\end{align}
where the $ij$ subscript indicates sampling conditional on $A_iA_j$, similar to before.

\section{Efficiency results} \label{scn:efficiency_results}

In this section we analyse the performance of the estimators in a rare-event
setting.  Recall that in such a setting, $\{A_1(\gamma),\dots,A_d(\gamma)\}$ denotes 
an indexed collection of not necessarily independent rare events and our objective is to
calculate $\alpha(\gamma)=\Prob(\bigcup_{i}^d A_i(\gamma))$ as $\gamma\to\infty$.
For such a \emph{rare-event} estimation problem there are specialised concepts of efficiency. In Section~\ref{scn:criteria} these definitions of efficiency are introduced. In addition,
we provide efficiency criteria for the proposed estimators under very general assumptions. 

In Sections~\ref{scn:ident_margs} and~\ref{scn:elliptical} we specialise in rare events associated with extremes. 
In such a framework, we show when the estimator $\hat{\alpha}_1$ is efficient for: i) 
a vast array of multivariate distributions with
identical marginals in Section~\ref{scn:ident_margs}, and ii) the specific cases of normal and elliptical distributions in Section~\ref{scn:elliptical}.
For this section we take the number of replicates $R$ to be 1.

\subsection{Variance Reduction} \label{scn:Vreduction}

First we compare the efficiency of our proposed estimator $\hat{\alpha}_1$ against that of
the crude Monte Carlo (CMC) estimator $\hat{\alpha}_0(\gamma)$ 
of $\alpha(\gamma):=\Prob(A(\gamma))$. 
An upper bound for $\Var \hat{\alpha}_0(\gamma)$ is
\begin{equation*}
	\Var \hat{\alpha}_0(\gamma) = \Prob(A(\gamma)) [1 - \Prob(A(\gamma)] < \Prob( A(\gamma) ) 
	\le \sum_i \Prob(A_i(\gamma)) \,.
\end{equation*}
This implies that the variance of the CMC estimator is of order 
$\Oh(\max_i \Prob(A_i(\gamma)))$,
which is the best possible without making any further assumptions.
In contrast an upper bound of $\Var \hat{\alpha}_1(\gamma) = \Var R_1$, where $R_1 =(1-E) \Ind\{E \ge 2\}$ from \eqref{residual}, is
\begin{equation}
	\Var\hat{\alpha}_1(\gamma) \le \Exp[R_1^2] < 2 \Exp \Big[ \BigBinom{E}{2} \Ind\{E \ge 2\} \Big] \underset{\eqref{summands}}{=} 2 \sum_{i<j} \Prob(A_i(\gamma), A_j(\gamma)) \,. \label{Bound}
\end{equation}
Thus the variance of our estimator $\hat{\alpha}_1(\gamma)$ is of order 
$\Oh(\max_{i<j} \Prob(A_i(\gamma), A_j(\gamma)))$,
so we can conclude that $\hat{\alpha}_1(\gamma)$ is asymptotically superior to CMC.

Next we turn our attention to the estimator $\hat{\beta}_n$.  
The following proposition shows that the reduction of variance of the estimator
$\hat{\beta}_n$ is of at least of a factor $\max_{|I|=n}\Prob(B_I)$
with respect to the non-conditional (crude) version estimator $\hat{\beta}_n^{[0]}$
defined as
\begin{equation}
 \hat{\beta}_n^{[0]}:=\sum_{|I|=n} \frac1{\lceil R / \binom{d}{n} \rceil} \sum_{r=1}^{\lceil R / \binom{d}{n} \rceil}  Y_{Ir}\Ind\{B_IC_I\}
\end{equation}
\begin{Proposition}\label{prop:VarRedCMC}
 \begin{equation*}
  \Var(\hat{\beta}_n) \le \max_{|I|=n}\Prob(B_I) 
   \Var(\hat{\beta}_n^{[0]}) \,.
 \end{equation*}
\end{Proposition}
\begin{proof}
Let $W_I:=Y\Ind\{C_I\}$.  By independence of the $W_I$
we can write the variance of $\hat{\beta}_n$ as
\begin{align*}
\Var(\hat{\beta}_n)
  =\Var\Big(\sum_{|I|=n} W_I \Prob(B_I) \,\Big|\, B_I \Big)
     &=\sum_{|I|=n} \Prob(B_I)^2 \Var( W_I \mid B_I )\\
     &\le\max_{|I|=n}\Prob(B_I)\sum_{|I|=n} \Prob(B_I)\Var( W_I \mid B_I )\,.
\end{align*}
Now, observe that
\begin{align*}
 \Prob(B_I)\Var( W_I \mid B_I )
   &\le\Prob(B_I)\Exp[ W_I^2 \mid B_I] -{\Prob(B_I)^2\Exp[W_I \mid B_I ]^2}\\
   &=\Exp[W_I^2\Ind\{B_I\}]-\Exp[W_I\Ind\{B_I\}]^2
   =\Var\left[W_I\Ind\{B_I\}\right].
\end{align*}
Thus we have proven that
\begin{align*}
\Var ( \hat{\beta}_n )
   &\le\max_{|I|=n}\Prob(B_I)\sum_{|I|=n} \Var( W_I\Ind\{B_I\} )
   =\max_{|I|=n}\Prob(B_I)\sum_{|I|=n}\Var(\hat{\beta}_0 )\,.
\end{align*}
\qed
\end{proof}

\subsection{Efficiency criteria} \label{scn:criteria}

We now ask if and when $\hat{\alpha}_1$ and $\hat{\beta}_n$ are efficient in the rare-event sense. We must first define efficiency, as there are several common benchmarks for the efficiency of a rare-event estimator.

\begin{Definition} \label{efficiency_criteria} An estimator $\hat{p}_{\gamma}$ of some rare probability $p_{\gamma}$ which satisfies $\forall \epsilon>0$ \\
\begin{subequations} \label{criteria}
 \begin{tabularx}{\hsize}{XXX}
     \begin{equation}
       \limsup\limits_{\gamma\to\infty} \, \frac{\Var\hat p_{\gamma}}
	{p_{\gamma}^{2-\epsilon}}=0 \label{a}
     \end{equation} &
     \begin{equation}
       \limsup\limits_{\gamma\to\infty}\frac{\Var \hat{p}_{\gamma}}
	{p_{\gamma}^2}<\infty \label{b}
     \end{equation} &
     \begin{equation}
       \limsup\limits_{\gamma\to\infty} \, \frac{\Var\hat{p}_{\gamma}}{p_{\gamma}^2}=0  \label{c}
     \end{equation}
   \end{tabularx}
\end{subequations}
has \emph{logarithmic efficiency} (LE) \eqref{a}, \emph{bounded relative error} (BRE) \eqref{b}, or \emph{vanishing relative error} (VRE) \eqref{c} respectively.
\end{Definition}

The levels of efficiency in Definition~\ref{efficiency_criteria} are given in increasing order of strength, that is, VRE $\Rightarrow$ BRE $\Rightarrow$ LE. As VRE is often too difficult a goal, we focus on BRE and LE. The following proposition gives an alternative form of the conditions in \eqref{criteria} for the specific case of our estimator $\hat{\alpha}_1$.

\begin{Proposition} \label{prop:efficiency}
	The estimator $\hat{\alpha}_1$ has LE iff it holds that $\forall \epsilon > 0$
	\begin{equation}\label{our_LE_criteria}
	  \limsup_{\gamma\to\infty}
		\dfrac{\max_{i < j} \, \Prob(A_i(\gamma), A_j(\gamma))}{\max_k \, \Prob(A_k(\gamma))^{2-\epsilon}} = 0 \,,
	\end{equation}
	and has BRE iff
	\begin{equation}\label{our_BRE_criteria}
	  \limsup_{\gamma\to\infty}
		\dfrac{\max_{i < j} \, \Prob(A_i(\gamma), A_j(\gamma))}{\max_k \, \Prob(A_k(\gamma))^2} < \infty \,.
	\end{equation}
\end{Proposition}
\begin{proof} 
	We prove the LE claim \eqref{our_LE_criteria}. Proof of the BRE claim follows the same arguments. \\
	($\Rightarrow$) We can see that
	\begin{align} 
		\Var \hat{\alpha}_1(\gamma) &\ge \Var \Ind\{E \ge 2\} = \Prob(E \ge 2) \, \Prob(E \le 1) \,, \quad \Prob(E \le 1)\to1 \,,  \label{first_bound} \\
	\intertext{and}		
		\Prob(E \ge 2) &\ge \BigBinom{d}{2}^{-1} \sum_{n=2}^d \BigBinom{n}{2} \Prob(E = n) \underset{\eqref{summands}}{=} \BigBinom{d}{2}^{-1} \sum_{i<j} \Prob(A_i(\gamma), A_j(\gamma)) \label{second_bound} \,.
	\end{align}
	So, $\forall \epsilon > 0$,
	\begin{align*}
		0 &\underset{\eqref{a}}{=} \limsup_{\gamma\to\infty}\,\frac{\Var\hat{\alpha}_1(\gamma)}{\Prob(A)^{2-\epsilon}} 
		\underset{\eqref{boole-frechet} \,\&\,\eqref{first_bound}}{>} \limsup_{\gamma\to\infty}\,\frac{\Prob(E \ge 2)}{\big(\sum_k \Prob(A_k(\gamma))\big)^{2-\epsilon}} \\
		&\underset{\eqref{second_bound}}{\ge} \Big[ d^{2-\epsilon} \BigBinom{d}{2} \Big]^{-1} \limsup_{\gamma\to\infty}\,\frac{\max_{i < j}\Prob(A_i(\gamma), A_j(\gamma))}{\max_k 
		 \Prob(A_k(\gamma))^{2-\epsilon}}
	\end{align*}
	which implies \eqref{our_LE_criteria}. \\
	($\Leftarrow$) We can see that, $\forall \epsilon > 0$,
	\begin{align*}
		\limsup_{\gamma\to\infty}\,\frac{\Var\hat{\alpha}_1(\gamma)}{\Prob(A)^{2-\epsilon}}
		&\underset{\eqref{boole-frechet} \,\&\, \eqref{Bound}}{<} \limsup_{\gamma\to\infty}\,\frac{2 \sum_{i<j}\Prob(A_i(\gamma), A_j(\gamma))}
		{(\max_k \, \Prob(A_k(\gamma)))^{2-\epsilon}} \\
		&\le 2\BigBinom{d}{2} \limsup_{\gamma\to\infty}\,\frac{\max_{i < j}
		 \Prob(A_i(\gamma), A_j(\gamma))}
		{\max_k \Prob(A_k(\gamma))^{2-\epsilon}} 
		\underset{\eqref{our_LE_criteria}}{=} 0 \,,
	\end{align*}
	which implies \eqref{a}. \qed
\end{proof}

\begin{Example}
   If the $A_i$ events are independent then the estimator $\hat{\alpha}_1$ has BRE.
\end{Example}

For the efficiency of our $\hat{\beta}_n$ estimators, 
the following proposition provides a very simple yet non-trivial condition for BRE.
\begin{Proposition}
 The estimator $\hat{\beta}_n(\gamma)$ has BRE if
 \begin{equation*}
  \limsup_{\gamma\to\infty}\frac{\max_{|I|=n}\Prob(B_I)}{\beta_n(\gamma)}<\infty.
 \end{equation*}
\end{Proposition}
\begin{proof}
By Proposition~\ref{prop:VarRedCMC} and the hypothesis we have
\begin{align*}
 \limsup_{\gamma\to\infty}\frac{\Var (\hat{\beta}_n(\gamma) )}{\beta_n^2(\gamma)}
 &\le  \limsup_{\gamma\to\infty}\frac{\max_{|I|=n}\Prob(B_I)\Var ( \hat{\beta}_n^{[0]}(\gamma) )}{\beta_n^2(\gamma)}\\
 &\le c\, \limsup_{\gamma\to\infty}\frac{\Var( \hat{\beta}_n^{[0]}(\gamma) )}{\beta_n(\gamma)} \,.
\end{align*}
Since $\hat{\beta}_n^{[0]}$ is an estimator in crude form
then $\Var( \hat{\beta}_n^{[0]}(\gamma) )=\Oh(\beta_n(\gamma))$ as $\gamma \to \infty$, so the proof
is complete. \qed
\end{proof}
\begin{Corollary}
 The estimator $\widehat{(\beta_1 \ddagger \alpha)}$ from \eqref{beta_on_alpha_1} has BRE.
\end{Corollary}

\subsection{Efficiency for identical marginals and dependence} \label{scn:ident_margs}

In this and the following subsections, we concentrate on rare events associated to 
extremes.  More precisely, we let $\bfX = (X_1,\dots,X_n)$ be an arbitrary random vector and define
$M=\max_i X_i$.  Therefore, we define $A_i(\gamma)=\{X_i> \gamma\}$ implying that the event of interest $A$ is equivalent to $\{M>\gamma\}$.

In this subsection, we assume the $X_i$ have identical marginal distributions. This simplifies the condition for BRE of $\hat{\alpha}_1$, \eqref{our_BRE_criteria}, so that it is now solely determined by the \emph{copula} of $\bfX$. We investigate some common tail dependence measures of copulas (tail dependence parameter and residual tail index) and also some common families of copulas (Archimedean copulas) to see when the estimator $\hat{\alpha}_1$ exhibits efficiency. 

\subsubsection{Asymptotic dependence} \label{ssec:asymp_dep}

The most basic measurement of tail dependence between a pair $(X_i, X_j)$ with common marginal distribution $F$ and copula $C_{ij}$ (cf.\ \cite{joe1997multivariate,nelsen2007introduction}) is
\[
	\lambda_{ij} = \lim_{\upsilon\to1} \Prob(X_i > \upsilon \mid X_j > \upsilon) = \lim_{\upsilon \to 1} \frac{1 - C_{ij}(\upsilon,\upsilon)}{1-\upsilon}
\]
where $\lambda_{ij} \in [0,1]$ is called the \emph{(upper) tail dependence parameter (or coefficient)} \cite{joe1997multivariate,mcneil2015quantitative}. We say the $(X_i,X_j)$ pair exhibit \emph{asymptotic independence} (AI) when $\lambda_{ij} = 0$, or \emph{asymptotic dependence} (AD) when $\lambda_{ij} > 0$. The canonical examples given for each case are the (non-degenerate) bivariate normal distribution for AI, and the bivariate Student $t$ distribution for AD \cite{sibuya1960bivariate}. 

For $\hat{\alpha}_1$ to have BRE, all pairs in $\bfX$ must exhibit AI. This is a necessary but not sufficient condition, therefore we will employ a more refined tail dependence measurement.

\subsubsection{Residual tail index}

We must first define two classes of functions:
\begin{itemize}
\item $L(x)$ is \emph{slowly-varying} (at $\infty$) if $L(cx)/L(x) \to 1$ as $x \to \infty$ for all $c > 0$,
\item $f(x)$ is \emph{regularly-varying} (at $\infty$) with index $\tau > 0$ if it takes the form $f(x) = L(x) x^{-\tau}$ for some $L(x)$ which is slowly-varying (cf.\ \cite{bingham1989regular,resnick1987extreme}).
\end{itemize}

We will assume, w.l.o.g., the marginals of $\bfX$ to be unit Fr\'echet distributed (i.e., $F_1(x) = \exp(-x^{-1}) \sim 1 - x^{-1}$). Ledford and Tawn \cite{ledford1996statistics,ledford1997modelling,ledford1998concomitant} first noted that the joint survivor functions for a wide array of bivariate distributions satisfy
\begin{equation} \label{ledford_form}
	\Prob(X_i > \gamma, X_j > \gamma) \sim L(\gamma) \gamma^{-1/\eta} \qquad \text{ as } \gamma \to \infty
\end{equation}
for a slowly-varying $L(\gamma)$ and an $\eta \in (0, 1]$. In other words, \eqref{ledford_form} says that $\Prob(X_i > \gamma, X_j > \gamma)$ is regularly-varying with index $1/\eta$.

The index is called the \emph{residual tail index} \cite{de2007extreme,nolde2014geometric}.\footnote{The older (and less insightful) name for $\eta$ is the \emph{coefficient of tail dependence} \cite{ledford1996statistics,resnick2002hidden}.} When $(X_i, X_j)$ exhibit AD (AI) then we typically have $\eta = 1$ ($\eta < 1$).\footnote{Hashorva \cite{hashorva2010residual} has found a case where an elliptically distributed $(X_i, X_j)$ has $\eta = 1$ and AI.} For independent components we have $\eta = 1/2$, so Ledford and Tawn \cite{ledford1996statistics} describe bivariate distributions with $\eta = 1/2$ as having \emph{near independence}. When $\eta < 1/2$ the random pair take large values together less frequently than they would if independent.

Returning to our original problem of estimating $\alpha(\gamma)$, let us label the residual tail index for every $(X_i,X_j)$ pair of $\bfX$ as $\eta_{ij}$. Also, let $\eta=\max_{ij}\eta_{ij}$ and $L$ be the associated slowly varying function. The following proposition outlines how these values relate to efficiency of $\hat{\alpha}_1$:

\begin{Proposition}
If \eqref{ledford_form} is satisfied for the maximal pair of $\bfX$, that is,
\[ \max_{i < j} \Prob(X_i > \gamma, X_j > \gamma) \sim L(\gamma) \gamma^{-1/\eta} \qquad \text{ as } \gamma \to \infty \,, \]
then the estimator $\hat{\alpha}_1$ has: i) BRE if $\eta<1/2$ or if $\eta=1/2$ and $L(\gamma) \not\to \infty$ as $\gamma \to \infty$, ii) LE if $\eta=1/2$.  
\end{Proposition}

\begin{proof}
Label the components of $\bfX$ such that
\[ \max_{i < j} \, \Prob(X_i > \gamma,X_j > \gamma) = \Prob(X_1 > \gamma,X_2 > \gamma) \]
then the condition for LE becomes, $\forall \epsilon > 0$
\[ \limsup_{\gamma \to \infty} \frac{\max_{i < j} \Prob(X_i\ge \gamma, X_j\ge \gamma)}{\max_k \Prob(X_k\ge \gamma)^{2-\epsilon}} =
 \limsup_{\gamma \to \infty} \frac{ L(\gamma) \gamma^{-1/\eta} } { (\gamma^{-1})^{2-\epsilon} } = 
 \limsup_{\gamma \to \infty} L(\gamma) \gamma^{2 - \frac1\eta - \epsilon} = 0 \]
which is equivalent to $\eta \in (0, 1/2]$; the $\eta=1/2$ case has LE as $\gamma^{-\epsilon} L(\gamma) \to 0$ for all $\epsilon > 0$ (see Proposition~1.3.6 part (v) of \cite{bingham1989regular}). Similarly we have BRE for $\eta \in (0, 1/2)$, but for the $\eta = 1/2$ case we also require that $L(\gamma) \not\to \infty$. \qed
\end{proof}

Heffernan \cite{heffernan2000directory} has conveniently compiled a directory of $\eta$ and $L(x)$ for many copulas which satisfy \eqref{ledford_form}. A summary of these results is given in Table \ref{tbl:etas_for_copulas}. In reading Heffernan's directory, one can spot two trends: normally $\eta \in \{1/2, 1\}$ and $L$ is a constant. The oft-cited Gaussian copula is the only exception for both of these trends in Heffernan's directory, having
$\eta = (1 + \rho)/2$ and $L(x) \propto (\log x)^{-\rho/(1+\rho)}$; Section~\ref{scn:elliptical} deals with the Gaussian case in detail.

\setlength\extrarowheight{3.5pt}

\begin{table}
\centering
\caption{Residual tail dependence index $\eta$ and $L(x)$ for various copulas. This is a subset of Table~1 of \protect{\cite{heffernan2000directory}} (their row numbers are preserved).} \label{tbl:etas_for_copulas}
\subfloat[Copulas with BRE.]{
\resizebox{0.45\textwidth}{!}{%
\centering
\begin{tabular}{|c|c|c|c|}
\hline
\bf{\#} & \bf{Name} & $\eta$ & $L(x)$ \\
\hline
1 & Ali-Mikhail-Haq & $0.5$ & $1+\tau$ \\
\hline
2 & BB10 in Joe & $0.5$ & $1+\theta/\tau$ \\
\hline
3 & Frank & $0.5$ & $\delta / (1-\e^{-\delta})$ \\
\hline
4 & Morgenstern & $0.5$ & $1+\tau$ \\
\hline
5 & Plackett & $0.5$ & $\delta$ \\
\hline
6 & Crowder & $0.5$ & $1 + (\theta-1)/\tau$ \\
\hline
7 & BB2 in Joe & $0.5$ & $\theta (\delta + 1) + 1$ \\
\hline
8 & Pareto & $0.5$ & $1 + \delta$ \\
\hline
9 & Raftery & $0.5$ & $\delta/(1-\delta)$ \\
\hline
\end{tabular}}}
\qquad
\subfloat[Copulas without BRE.]{
\resizebox{.45\textwidth}{!}{%
\begin{tabular}{|c|c|c|c|}
\hline
\bf{\#} & \bf{Name} & $\eta$ & $L(x)$ \\
\hline
11 & Joe & $1$ & $2-2^{1/\delta}$ \\
\hline
12 & BB8 in Joe & $1$ & $2-2(1-\delta)^{\theta-1}$ \\
\hline
13 & BB6 in Joe & $1$ & $2-2^{1/(\delta\theta)}$ \\
\hline
14 & Extreme value & $1$ & $2-V(1,1)$ \\
\hline
15 & B11 in Joe & $1$ & $\delta$ \\
\hline
16 & BB1 in Joe & $1$ & $2-2^{1/\delta}$ \\
\hline
17 & BB3 in Joe & $1$ & $2-2^{1/\theta}$ \\
\hline
18 & BB4 in Joe & $1$ & $2^{-1/\delta}$ \\
\hline
19 & BB7 in Joe & $1$ & $2-2^{1/\theta}$ \\
\hline
\end{tabular}}}
\end{table}

\subsubsection{Archimedean Copulas}

Some of the most frequently used copulas are in the family of \emph{Archimedean copulas}. These are very general models and  are widely used in applications due to their flexibility. A copula is Archimedean if there exists a function $\psi$ such that the copula $C$ can be written as
\begin{equation*}
	C(u_1,\dots,u_n)=\psi^{\leftarrow}(\psi(u_1)+\dots+\psi(u_n)).
\end{equation*}
The function $\psi$, called the \emph{generator} of the copula, defines a copula if its functional inverse is the Laplace transform of a non-negative random variable.
For Archimedean copulas we can restate the BRE condition \eqref{our_BRE_criteria} in terms of the generator $\psi$.

\begin{Theorem}[Thm.\ 3.4 of \cite{charpentier2009tails}] \label{arch_cop_effic}
Let $(U_1, \dots, U_n) \sim C$ where $C$ is an Archimedean copula with generator $\psi$. If $\psi^{\leftarrow}$ is twice continuously differentiable and its second derivative is bounded at 0 then $\forall \, i \not= j$
\[ \lim_{u \to 0} \frac{\Prob(U_i \geq 1 - u x_1, U_j \geq 1 - u x_2)}{u^2} < \infty  \]
for any $0 < x_1,\, x_2 < \infty$.
\end{Theorem}

\begin{Corollary} \label{cor:arch_effic}
Consider using $\hat{\alpha}_1$ for a distribution with common marginal distributions and a copula $C$. If $C$ satisfies the conditions of \thrm{arch_cop_effic} then $\hat{\alpha}_1$ has BRE.
\end{Corollary}

Charpentier and Segers \cite{charpentier2009tails} have helpfully created a directory of Archimedean copulas from which we can see if the BRE conditions from Corollary~\ref{cor:arch_effic} are satisfied. Using this information, we provide a summary of the efficiency status of many Archimedean copulas in Table~\ref{effic_arch_cops}.

\begin{table}
\centering
\caption{Examples of Archimedean copula families. Names (if they are named) and generator functions are listed, as are the ranges for which $\theta$ is valid and the subset of $\theta$ which ensures that $\hat{\alpha}_1$ has BRE. A $\Theta$ in the final column means that all valid $\theta$ ensure BRE. The families listed appear in Table 4.1 of \protect{\cite{nelsen2007introduction}} and Table 1 of \protect{\cite{charpentier2009tails}}.} \label{effic_arch_cops}
\vspace{1em}
\resizebox{.75\textwidth}{!}{%
\begin{tabular}{|c|c|c|c|c|}
\hline
\bf{\#} & \bf{Name} & \bf{Generator} $\psi(t)$ & \bf{Valid} $\theta$ & \bf{Efficient} $\theta$ \\[3pt]
\hline
1 & Clayton & $\frac{1}{\theta}(t^{-\theta} - 1)$ & $[-1, \infty)$ & $\Theta$ \\[3pt]
\hline
2 & \cellcolor{lightgray} & $(1-t)^\theta$ & $[1, \infty)$ & $\{1\}$ \\[3pt]
\hline
3 & Ali--Mikhail--Haq & $\log \frac{1-\theta(1-t)}{t}$ & $[-1,1)$ & $\Theta$ \\[3pt]
\hline
4 & Gumbel--Hougaard & $({-}\log t)^\theta$ & $[1,\infty)$ & $\{1\}$ \\[3pt]
\hline
5 & Frank & ${-}\log \frac{\e^{-\theta t} - 1}{\e^{-\theta} - 1}$ & $\RL$ & $\Theta \setminus \{0\}$ \\[3pt]
\hline
6 & \cellcolor{lightgray} & ${-}\log [1 - (1-t)^\theta]$ & $[1, \infty)$ & $\{ 1\}$ \\[3pt]
\hline
7 & \cellcolor{lightgray} & ${-}\log[\theta t + (1-\theta)]$ & $(0, 1]$ & $\Theta$ \\[3pt]
\hline
8 & \cellcolor{lightgray} & $\frac{1-t}{1+(\theta-1)t}$ & $[1, \infty)$ & $\Theta$ \\[3pt]
\hline
9 & \cellcolor{lightgray} & $\log(1-\theta \log t)$ & $(0,1]$ & $\Theta$ \\[3pt]
\hline
10 & \cellcolor{lightgray} & $\log(2t^{-\theta} - 1) $ & $(0,1]$ & $\Theta$ \\[3pt]
\hline
11 & \cellcolor{lightgray} & $\log(2-t^{\theta})$ & $(0,1/2]$ & $\Theta$ \\[3pt]
\hline
12 & \cellcolor{lightgray} & $(\frac1t - 1)^\theta $ & $[1,\infty)$ & $\{1\}$ \\[3pt]
\hline
13 & \cellcolor{lightgray} & $(1-\log t)^\theta - 1$ & $(0,\infty)$ & $\Theta$ \\[3pt]
\hline
14 & \cellcolor{lightgray} & $(t^{-1/\theta} - 1)^\theta$ & $[1, \infty)$ & $\{1\}$ \\[3pt]
\hline
15 & \cellcolor{lightgray} & $(1 - t^{1/\theta})^\theta$ & $[1, \infty)$ & $\{1\}$ \\[3pt]
\hline
16 & \cellcolor{lightgray} & $(\frac{\theta}{t} + 1)(1-t)$ & $[0, \infty)$ & $\Theta$ \\[3pt]
\hline
17 & \cellcolor{lightgray} & $-\log \frac{(1+t)^{-\theta}-1}{2^{-\theta} - 1}$ & $\RL$ & $\Theta \setminus \{0\}$ \\[3pt]
\hline
18 & \cellcolor{lightgray} & $\e^{\theta/(t-1)}$ & $[2, \infty)$ & $\emptyset$ \\[3pt]
\hline
19 & \cellcolor{lightgray} & $\e^{\theta/t} - \e^{\theta}$ & $(0, \infty)$ & $\Theta$ \\[3pt]
\hline
20 & \cellcolor{lightgray} & $\e^{t^{-\theta}} - \e$ & $(0, \infty)$ & $\Theta$ \\[3pt]
\hline
21 & \cellcolor{lightgray} & $1 - [1 - (1-t)^\theta]^{1/\theta}$ & $[1, \infty)$ & $\{1\}$ \\[3pt]
\hline
22 & \cellcolor{lightgray} & $\arcsin(1-t^\theta)$ & $(0, 1]$ & $\Theta$ \\[3pt]
\hline
\end{tabular}}
\end{table}

The efficiency of $\hat{\alpha}_1$ can be proved without the assumption of identical marginal distributions, but the efficiency must be shown case-by-case for each family of distributions. The next section does this for the multivariate normal distribution and for some elliptical distributions. 

\subsection{Efficiency for the case of normal and elliptical distributions} \label{scn:elliptical}

The efficiency characteristics of normally and elliptically distributed random vectors are very similar. This section defines these distributions, outlines their asymptotic properties, then shows the conditions in which $\hat{\alpha}_1$ exhibits levels of asymptotic efficiency.

\subsubsection{Definitions and categories of elliptical distributions}

Let $\NormDist_d(\bfmu,\bfSigma)$ denote the multivariate normal distribution with mean $\bfmu \in \RL^d$ and positive-definite covariance matrix $\bfSigma \in \RL^{d \times d}$. Denote the corresponding density $\phi_{\bfmu, \bfSigma}(\cdot,\cdot)$, and write $\sigma_i^2 := \bfSigma_{ii}$, $\rho_{ij} := \bfSigma_{ij}/(\sigma_i \sigma_j)$. 
The normal distribution belong to the class of \emph{elliptical distributions},
which we denote $\EllDist(\bfmu, \bfSigma, F)$, where $F$ is the c.d.f.\ of a positive r.v. We define $\bfX \sim \EllDist(\bfmu, \bfSigma, F)$ as
\begin{equation}\label{ellip_distr}
	\bfX \eqdistr \bfmu + R \, \bfC \, U
\end{equation}
where $R \sim F$ is called the \emph{radial component}, $U$ is (independent of $R$ and) distributed uniformly on the $d$-dimensional unit hypersphere, and $\bfC \in \RL^{d \times d}$ satisfies $\bfC \bfC^\top = \bfSigma$. 
For background on elliptical distributions, see \cite{MR1990662}.
The efficiency of $\hat{\alpha}_1$ turns out to be related with
max-domain of attraction (MDA) of the radial 
component. The MDA is known from standard extreme value theory, see \cite{de2007extreme}.

We consider some subclasses of elliptical distributions depending on the MDA 
of the radial distribution:
\begin{itemize}
\item $F \in$ MDA(Fr\'echet), then Theorem 4.3 of \cite{hult2002multivariate} implies that $\bfX$ has asymptotic dependence and $\hat{\alpha}_1$ is never efficient (see Section~\ref{ssec:asymp_dep}).
\item $F \in$ MDA(Weibull), then components of $\bfX$ are light-tailed and uninteresting (in a rare-event context).
\item $F \in$ MDA(Gumbel), this is the interesting case which includes the normal distribution. Hashorva \cite{hashorva2007asymptotic} label these the \emph{type I elliptical random vectors}.
\end{itemize}

\subsubsection{Efficiency for type I elliptical distributions} \label{scn:type_1}

Take $\bfX \sim \EllDist(\bfmu, \bfSigma, F)$ where the radial distribution $F \in$  MDA(Gumbel) has support $(0,x_F)$, for some $x_F \in \overline{\RL}$, and where $\{\sigma_1, \dots, \sigma_d\}$ are in decreasing order. By definition of the Gumbel MDA, one can find a scaling function $w(x)$ satisfying
\begin{equation*}
  \lim_{x\to x_F} \frac{\Ftail(x+t/w(x))}{\Ftail(x)}=\e^{-t}.
\end{equation*}
One frequently takes $w(x):= \Ftail(x) / \int_{x}^{x_F}\Ftail(s) \dd s$.
Also, define $\upsilon_i(\gamma) := (\gamma - \mu_i)/\sigma_i$ and $a_{ij} := \sigma_j/\sigma_i$. If $\rho_{ij} \ge a_{ij}$ then set
 \[ \mu_{ij}:= \mu_j \quad \text{ and } \quad \kappa_{ij}:= \sigma_j \]
 otherwise for $\rho_{ij} < a_{ij}$
 \[ \mu_{ij}:= \frac{\mu_i-a_{ij}\rho_{ij}(\mu_1+\mu_2)+a^2\mu_j}{\alpha_{ij}(1-\rho^2_{ij})} \quad \text{ and } \quad \kappa_{ij}:= \frac{\sigma_i^2\sigma_j^2(1-\rho^2_{ij})}{ {\sigma_i^2-2\rho_{ij}\sigma_i\sigma_j + \sigma_{j}^2}} \,. \]

We now apply the asymptotic properties outlined in the
Appendix to assess the efficiency of $\hat{\alpha}_1$ for type I elliptical distributions. 

\begin{Theorem} \label{thm:ellip_efficiency}
Consider $\bfX \sim \EllDist(\bfmu, \bfSigma, F)$ where $F \in$ MDA(Gumbel), and let
 \[ \kappa := \max_{i<j} \kappa_{ij} \,, \quad \mu := \max_{i<j \,:\, \kappa = \kappa_{ij}} \mu_{ij} \,, \quad \text{and} \quad  \upsilon(\gamma):= (\gamma-\mu)/\kappa + \oh(1) \,. \]
 If $\kappa>\sigma_1$,\footnote{This implies that the Savage condition (see Appendix) is fulfilled at least for one pair.} then $\hat{\alpha}_1$ has LE if
 \begin{equation} \label{eq:type_1_effic}
 	\forall \epsilon>0 \quad 
   \limsup_{\gamma\to x_F}\frac{w(\upsilon(\gamma))\Ftail(\upsilon(\gamma))}
   {w(\upsilon_1(\gamma))\Ftail(\upsilon_1(\gamma))^{2-\epsilon}}<\infty \,.
 \end{equation}
 Moreover, if \eqref{eq:type_1_effic} holds for $\epsilon=0$ then $\hat{\alpha}_1$ has BRE.
\end{Theorem}

\begin{proof}
 It follows from \eqref{berman} and Theorem~\ref{thm:type_1_asymptotics} in the Appendix. 
 \qed
\end{proof}


\begin{Example}[Kotz Type III]
 One family of type I elliptical distributions, is the \emph{Kotz Type III} distributions, defined by
 \[
  \Ftail(\gamma)=(K+\oh(1))\gamma^N\exp(-ru^\delta),\quad w(\gamma)=r\delta \gamma^{\delta-1}\,, \quad \text{for } \gamma>0,
 \]
 with $K,\delta,N>0$.  In this case it is clear that
 \[
  \lim_{\gamma\to \infty}\frac{w(\upsilon(\gamma))}{w(\upsilon_1(\gamma))}=\left(\frac{\sigma_1}{\kappa}\right)^{\delta-1}<\infty,
 \]
 while
  \begin{align*}
  &\limsup_{\gamma\to \infty}\frac{\Ftail(\upsilon(\gamma))}{\Ftail(\upsilon_1(\gamma))^{2}} \\
   =&\limsup_{\gamma\to \infty}\Big(\frac{\sigma_1^2}{\kappa \gamma}\Big)^{N}
    \exp\Big\{ {-}r\Big(\Big(\frac{\gamma-\mu}{\kappa}\Big)^\delta-2\Big(\frac{\gamma-\mu_1}{\sigma_1}\Big)^\delta\Big)\Big\},\\
   =&\limsup_{\gamma\to \infty}\Big(\frac{\sigma_1^2}{\kappa \gamma}\Big)^{N}
    \exp\Big\{ {-}r\Big(\frac{\gamma^\delta-\delta\mu \gamma^{\delta-1}+\oh(\gamma^{\delta-1})}{\kappa^\delta}
    -\frac{\gamma^\delta-\delta\mu_1 \gamma^{\delta-1}+\oh(\gamma^{\delta-1})}{\sigma_1^\delta/2}\Big)\Big\} \,.
 \end{align*}
 Hence, $\hat{\alpha}_1$ has BRE in the following cases
 \begin{itemize}
  \item $\sigma_1^\delta>2\kappa^\delta$, or
  \item $\sigma_1^\delta=2\kappa^\delta$, $\delta>1$ and $\mu_1>\mu$.
 \end{itemize}
 The estimator $\hat{\alpha}_1$ has LE if $\sigma_1^\delta=2\kappa^\delta$, and is inefficient when $\sigma_1^\delta<2\kappa^\delta$.
\end{Example}

\begin{Example}[Normal distributions] \label{ex:normal}
The normal distribution is a Kotz III type distribution with $\delta=2$.
Hence, $\hat{\alpha}_1$ has BRE if $\sigma_1^2>2\kappa^2$, or $\sigma_1^2=2\kappa^2$ and $\mu_1>\mu$.
The estimator $\hat{\alpha}_1$ has LE if $\sigma_1^2=2\kappa^2$, and is inefficient when $\sigma_1^2<2\kappa^2$.
\end{Example}

Frequently, a set of random variables represents as a stochastic process $\{X_n\}_{n\ge0}\,$. The value of $\Prob(M > \gamma)$, with $M := \max_{1 \le n \le d} X_n$, in such cases usually valuable. The simplest case to take is when all $X_n$ have identical marginals such as in stationary processes; one such example is the autoregressive (AR) process.

\begin{Example}[AR(1) processes]
Say $X_t = \varphi X_{t-1} + \epsilon_t$, where $|\varphi| < 1$ and $\epsilon_t$ are i.i.d.\ $\NormDist_1(0, \sigma_{\epsilon}^2)$, and we start the process in stationarity. We have that each $X_i$ has the same marginal distribution, $X_i \sim \NormDist_1(0, \sigma_{\epsilon}^2/(1-\varphi^2))$, and
\[
	\max_{i < j} \, \Prob(X_i > \gamma, X_j > \gamma) = \begin{cases}
		\Prob(X_{\bullet} > \gamma, X_{\bullet+1} > \gamma) & \text{ if } \varphi > 0 \\
		\Prob(X_{\bullet} > \gamma, X_{\bullet+2} > \gamma) & \text{ if } \varphi < 0 \\
		\Prob(X_{\bullet} > \gamma)^2  & \text{ if } \varphi = 0
	\end{cases} \,.
\]
For $\varphi \not=0$ we know that
\[ (X_{\bullet+1} \mid X_\bullet = \gamma) \sim \NormDist_1(\varphi \gamma, \sigma_{\epsilon}^2), \text{ and } (X_{\bullet+2} \mid X_{\bullet} = \gamma) \sim \NormDist_1(\varphi^2 \gamma, \sigma_{\epsilon}^2 (1-\varphi^4) / (1-\varphi^2)) \,. \]
When $\varphi = 0$ the $X_i$ are independent and $\hat{\alpha}_1$ is trivially efficient, and when $\varphi \in (-1, 1) \setminus \{0\}$ we have (noting that $\{X_\bullet > \gamma\} \to \{X_\bullet = \gamma\}$) that
\begin{align*}
	\lim_{\gamma \to \infty} \frac{ \max_{i < j} \Prob(X_i > \gamma, X_j > \gamma) }{ \max_{i} \Prob(X_i > \gamma)^2 }
	&= \lim_{\gamma \to \infty} \frac{ \Prob(X_{\bullet} > \gamma, X_{\bullet\,+\,(1\,\mathrm{or}\,2)} > \gamma) }{ \Prob(X_{\bullet} > \gamma)^2 } \\
	&= \lim_{\gamma \to \infty} \frac{ \Prob(X_{\bullet\,+\,(1\,\mathrm{or}\,2)} > \gamma \mid X_{\bullet} = \gamma) }{ \Prob(X_{\bullet} > \gamma) } \\
	&= 0
\end{align*}
as $\sigma_{\epsilon}^2 < \sigma_{\epsilon}^2 (1-\varphi^4) / (1-\varphi^2)) < \sigma_{\epsilon}^2/(1-\varphi^2)$. Therefore, we have BRE of $\hat{\alpha}_1$ for all stationary AR(1) processes.
\end{Example}

\section{Numerical experiments} \label{scn:numerical_results}

We explore the performance of the estimators for the problem of $\Prob(M > \gamma)$ for $M = \max_i X_i$, where $\bfX$ is multivariate normal and multivariate Laplace distributed. The following notation is used: $\bfX_{-i}$ ($\bfX_{-i,-j}$) is the random vector $\bfX$ with $X_i$ ($X_i$ and $X_j$) removed, $\bfzero$ is the vector of zeros, $\bfI$ is the identity matrix, $\bfx^\ast$ is the transpose of $\bfx$, and $X \independent Y$ means $X$ and $Y$ are independent. We use some standard distributions: $\ExpDist(\lambda)$ for exponential ($f(x)\propto \e^{-\lambda x}$), $\IGDist(\mu,\lambda)$ for inverse Gaussian ($f(x) \propto x^{-3/2} \e^{-\lambda(x-\mu)^2/(2\mu^2x)}$), $\LaplaceDist$ for Laplace (defined in Case 2 below).
The \textsc{Matlab} and \textsc{Mathematica} code used to generate them are available online \cite{OnlineAccomp}.  

\subsection*{Case 1: Multivariate Normal distributions}

Let $\bfX \sim \NormDist_d(\bfzero, \bfSigma)$ where $\bfSigma = (1-\rho) \bfI + \rho$; that is, each $X_i \sim \NormDist_1(0,1)$ and $\Corr(X_i,X_j)=\rho$. We implement the first- and second-order IS regimes. The necessary conditional distributions are well-known and simple; both $\bfX_{-i} \mid X_i$ and $\bfX_{-i,-j} \mid (X_i, X_j)$ are normally distributed \cite{anderson1962introduction}.
Sampling from $X_i \mid X_i > \gamma$ can be easily done by acceptance--rejection with shifted exponential proposals \cite{robert1995simulation} (or by inverse transform sampling \cite[Remark 2.4]{asmussen2007stochastic}, though this can be problematic using only double precision arithmetic). To simulate $(X_i,X_j) \mid \min\{X_i, X_j\}>\gamma$ we use Botev's \textsc{Matlab} library \cite{botev2017normal}, but also remark that a Gibb's sampler is a commonly used alternative \cite{breslaw1994random,robert1995simulation}.

\subsection*{Case 2: Multivariate Laplace distributions} \label{scn:laplace}

Let $\bfX \sim \LaplaceDist$. We can define this distribution by
\[ \bfX \eqdistr \sqrt{R} \bfY \,, \quad \text{where } \bfY \sim \NormDist_d(\bfzero,\bfI), R \sim \ExpDist(1), \bfY \independent R \,.
\]
The distribution has been applied in a financial context \cite{huang2003rare}, and is examined in
\cite{eltoft2006multivariate,kotz2001asymmetric}. From the former we have that the density of $\LaplaceDist$ is
\[
f_{\bfX}(\bfx) = 2 (2 \pi)^{-d/2}  K_{(d/2) -1} \big( \sqrt{2 \bfx^\ast \bfx} \big) \,
 \big(\sqrt{\tfrac{1}{2} \bfx^\ast \bfx}\big)^{1-(d/2)}
\]
where $K_n(\cdot)$ denotes the modified Bessel function of the second kind of order $n$. 

To implement the first-order IS algorithm we need the conditional distributions $X_i \mid X_i > \gamma$ and
$\bfX_{-i} \mid \bfX_i$.
Assuming $\gamma > 0$ we can derive that $(X_i \mid X_i > \gamma) \sim \ExpDist(\sqrt{2})$.
Further calculation gives
\[
\bfX_{-1} \mid X_1 \ \eqdistr\ \frac{X_1}{Y_1} \bfY_{-1} \mid (\sqrt{R} Y_1 = X_1) \ \eqdistr\
\frac{X_1}{Y_{1,X_1}} \bfY_{-1} \,,
\]
where $Y_{1,X_1} \sim (Y_1 \mid \sqrt{R} Y_1 = X_1)$, noting that $Y_{1,X_1} \independent \bfY_{-1}$ because of the independence between the entries of $\bfY$. Direct calculation gives 
\[
f_{Y_i \mid \sqrt{R} Y_i}(y_i \mid x_i) = 
 2 \left| y_i\right|  \exp\left\{
   -x_i^2 / y_i^2 - x_i^2 / 2 + \sqrt{2} \left| x_i\right| \right\} / (\sqrt{\pi } y_i^2)
\]
which is the density of $\sqrt{X}$ where 
$X \sim \IGDist(\sqrt{2} |x_i|, 2 x_i^2)$.
This is summarised in the following algorithm.

\begin{algorithm}
\caption{Sampling $\bfX_{-i} \mid X_i > \gamma$ for the Laplace distribution}
\label{alg:cond_normals}
\begin{algorithmic}[1]
\State $X_i \gets \ExpDist(\sqrt{2})$
\State $Y_{i,X_i} \gets \IGDist(\sqrt{2} |X_i|,2 X_i^2)$.
\State $\bfY_{-i} \gets \NormDist_{d-1}(\mathbf{0},\bfI_{p-1})$.
\State \textbf{return} $X_i \bfY_{-i} / Y_{i,X_i}$.
\end{algorithmic}
\end{algorithm}

\subsection{Test setup}

The estimators tested are $\hat{\alpha}_0$ (crude Monte Carlo) and 
 $\hat{\alpha}_1$, $\hat{\alpha}_2$, $\hat{\alpha}_1^{[1]}$, $\hat{\alpha}_2^{[2]}$, $\widehat{(\beta_1 \ddagger \alpha)}$, $\widehat{(\beta_2 \ddagger \alpha)}$, defined in \eqref{alpha_1}, \eqref{alpha_2}, \eqref{alpha_1_IS1}, \eqref{alpha_2_IS_2}, \eqref{beta_on_alpha_1} and \eqref{beta_on_alpha_2} respectively. As a reference, we show the true value $\alpha$ (calculated by numerical integration using \textsc{Mathematica}), and the first two truncations of the IEF:
$ \alphaMax(\gamma) := \sum_i \Prob(X_i > \gamma)$ and $\alphaMax(\gamma){-}q(\gamma) := \sum_i \Prob(X_i > \gamma) - \sum_{i<j} \Prob(X_i > \gamma, X_j > \gamma)$.
Each estimator is given $R=10^6$, and an asterisk is placed in table entries where the corresponding estimate had 0 variance (i.e., the estimator had degenerated).

\newpage
\subsection{Results}
\enlargethispage{1cm}

\begin{table}[H]
\centering
\begin{tabular}{ccccc}
\toprule
\multirow{2}{*}{Estimators} & \multicolumn{4}{c}{$\gamma$} \\
                            & 2    & 4    & 6    & 8    \\ 
\midrule
$\alpha$ & 5.633e-02 & 1.095e-04 & 3.838e-09 & 2.481e-15 \\
\rowcolor{Gray} $\hat{\alpha}_0$ & 5.651e-02 & 1.140e-04 & 0* & 0* \\
$\alphaMax$ & 9.100e-02 & 1.267e-04 & 3.946e-09 & 2.488e-15 \\
\rowcolor{Gray} $\alphaMax{-}q$ & 4.000e-02 & 1.055e-04 & 3.827e-09 & 2.480e-15 \\
$\hat{\alpha}_1$ & 5.650e-02 & 1.047e-04 & 3.946e-09* & 2.488e-15* \\
\rowcolor{Gray} $\hat{\alpha}_2$ & 5.605e-02 & 1.075e-04 & 3.827e-09* & 2.480e-15* \\
$\hat{\alpha}_1^{[1]}$ & 5.637e-02 & 1.096e-04 & 3.837e-09 & 2.481e-15 \\
\rowcolor{Gray} $\hat{\alpha}_2^{[2]}$ & 5.633e-02 & 1.095e-04 & 3.838e-09 & 2.481e-15 \\
$\widehat{(\beta_1 \ddagger \alpha)}$ & 5.634e-02 & 1.095e-04 & 3.838e-09 & 2.480e-15 \\
\rowcolor{Gray} $\widehat{(\beta_2 \ddagger \alpha)}$ & 5.631e-02 & 1.095e-04 & 3.838e-09 & 2.481e-15 \\
\bottomrule
\end{tabular}
\caption{Estimates of $\Prob(M>\gamma)$ where $M=\max_i X_i$ and $\bfX \sim \NormDist_4(\bfzero_4, \bfSigma)$, $\rho = 0.75$.}
\label{table:norm_ests}
\end{table}

\vspace{-2em}

\begin{table}[H]
\centering
\begin{tabular}{ccccc}
\toprule
\multirow{2}{*}{Estimators} & \multicolumn{4}{c}{$\gamma$} \\
                            & 2    & 4    & 6    & 8    \\ 
\midrule
$\hat{\alpha}_0$ & 3.109e-03 & 4.075e-02 & 1* & 1* \\
\rowcolor{Gray} $\alphaMax$ & 6.154e-01 & 1.566e-01 & 2.822e-02 & 3.142e-03 \\
$\alphaMax{-}q$ & 2.899e-01 & 3.665e-02 & 2.827e-03 & 1.147e-04 \\
\rowcolor{Gray} $\hat{\alpha}_1$ & 2.977e-03 & 4.429e-02 & 2.822e-02* & 3.142e-03* \\
$\hat{\alpha}_2$ & 5.077e-03 & 1.839e-02 & 2.827e-03* & 1.147e-04* \\
\rowcolor{Gray} $\hat{\alpha}_1^{[1]}$ & 6.918e-04 & 4.639e-04 & 1.747e-04 & 2.192e-05 \\
$\hat{\alpha}_2^{[2]}$ & 7.838e-08 & 8.647e-05 & 1.237e-05 & 4.010e-08 \\
\rowcolor{Gray} $\widehat{(\beta_1 \ddagger \alpha)}$ & 6.564e-05 & 7.046e-05 & 6.227e-05 & 4.362e-05 \\
$\widehat{(\beta_2 \ddagger \alpha)}$ & 3.493e-04 & 1.593e-05 & 6.883e-06 & 3.340e-07 \\
\bottomrule
\end{tabular}
\caption{Absolute relative errors of the estimates in Table~\ref{table:norm_ests}.}
\end{table}

\vspace{-2em}

\begin{table}[H]
\centering
\begin{tabular}{ccccc}
\toprule
\multirow{2}{*}{Estimators} & \multicolumn{4}{c}{$\gamma$} \\
                            & 2    & 4    & 6    & 8    \\ 
\midrule
$\hat{\alpha}_0$ & 2.309e-01 & 1.068e-02 & 0 & 0 \\
\rowcolor{Gray} $\hat{\alpha}_1$ & 2.557e-01 & 5.099e-03 & 0 & 0 \\
$\hat{\alpha}_2$ & 1.885e-01 & 1.414e-03 & 0 & 0 \\
\rowcolor{Gray} $\hat{\alpha}_1^{[1]}$ & 2.817e-02 & 3.071e-05 & 4.650e-10 & 9.972e-17 \\
$\hat{\alpha}_2^{[2]}$ & 9.901e-03 & 4.244e-06 & 1.908e-11 & 8.575e-19 \\
\rowcolor{Gray} $\widehat{(\beta_1 \ddagger \alpha)}$ & 1.929e-02 & 2.089e-05 & 3.197e-10 & 6.994e-17 \\
$\widehat{(\beta_2 \ddagger \alpha)}$ & 1.306e-02 & 5.265e-06 & 2.310e-11 & 1.035e-18 \\
\bottomrule
\end{tabular}
\caption{Standard deviations of the estimates in Table~\ref{table:norm_ests}.}
\end{table}

\begin{table}[H]
\centering
\begin{tabular}{ccccc}
\toprule
\multirow{2}{*}{Estimators} & \multicolumn{4}{c}{$\gamma$} \\
                            & 6    & 8    & 10    & 12    \\ 
\midrule
$\alpha$ & 4.093e-04 & 2.435e-05 & 1.442e-06 & 8.526e-08 \\
\rowcolor{Gray} $\hat{\alpha}_0$ & 3.910e-04 & 2.000e-05 & 2.000e-06 & 0* \\
$\alphaMax$ & 4.130e-04 & 2.441e-05 & 1.443e-06 & 8.527e-08 \\
\rowcolor{Gray} $\alphaMax{-}q$ & 4.093e-04 & 2.435e-05 & 1.442e-06 & 8.526e-08 \\
$\hat{\alpha}_1$ & 4.120e-04 & 2.441e-05* & 1.443e-06* & 8.527e-08* \\
\rowcolor{Gray} $\hat{\alpha}_2$ & 4.093e-04* & 2.435e-05* & 1.442e-06* & 8.526e-08* \\
$\hat{\alpha}_1^{[1]}$ & 4.093e-04 & 2.435e-05 & 1.442e-06 & 8.526e-08 \\
\rowcolor{Gray} $\widehat{(\beta_1 \ddagger \alpha)}$ & 4.093e-04 & 2.435e-05 & 1.442e-06 & 8.526e-08 \\
\bottomrule
\end{tabular}
\caption{Estimates of $\Prob(M>\gamma)$ where $M=\max_i X_i$ and $\bfX \sim \LaplaceDist$, $d=4$.}
\label{table:laplace_ests}
\end{table}

\begin{table}[H]
\centering
\begin{tabular}{ccccc}
\toprule
\multirow{2}{*}{Estimators} & \multicolumn{4}{c}{$\gamma$} \\
                            & 6    & 8    & 10    & 12    \\ 
\midrule
$\hat{\alpha}_0$ & 4.472e-02 & 1.786e-01 & 3.873e-01 & 1* \\
\rowcolor{Gray} $\alphaMax$ & 8.959e-03 & 2.473e-03 & 6.987e-04 & 2.003e-04 \\
$\alphaMax{-}q$ & 8.067e-05 & 8.266e-06 & 8.757e-07 & 9.506e-08 \\
\rowcolor{Gray} $\hat{\alpha}_1$ & 6.516e-03 & 2.473e-03* & 6.987e-04* & 2.003e-04* \\
$\hat{\alpha}_2$ & 8.067e-05* & 8.266e-06* & 8.757e-07* & 9.506e-08* \\
\rowcolor{Gray} $\hat{\alpha}_1^{[1]}$ & 8.470e-06 & 1.023e-05 & 3.019e-05 & 1.577e-05 \\
$\widehat{(\beta_1 \ddagger \alpha)}$ & 4.515e-05 & 2.948e-05 & 2.151e-06 & 2.833e-06 \\
\bottomrule
\end{tabular}
\caption{Absolute relative errors of the estimates in Table~\ref{table:laplace_ests}.}
\end{table}

\begin{table}[H]
\centering
\begin{tabular}{ccccc}
\toprule
\multirow{2}{*}{Estimators} & \multicolumn{4}{c}{$\gamma$} \\
                            & 6    & 8    & 10    & 12    \\ 
\midrule
$\hat{\alpha}_0$ & 1.977e-02 & 4.472e-03 & 1.414e-03 & 0 \\
\rowcolor{Gray}  $\hat{\alpha}_1$ & 1.000e-03 & 0 & 0 & 0 \\
$\hat{\alpha}_2$ & 0 & 0 & 0 & 0 \\
\rowcolor{Gray} $\hat{\alpha}_1^{[1]}$ & 2.735e-05 & 8.581e-07 & 2.752e-08 & 8.189e-10 \\
$\widehat{(\beta_1 \ddagger \alpha)}$ & 1.937e-05 & 6.086e-07 & 1.908e-08 & 5.990e-10 \\
\bottomrule
\end{tabular}
\caption{Standard deviations of the estimates in Table~\ref{table:laplace_ests}.}
\end{table}

\newpage
\subsection{Discussion}

We begin with some trends which we expected to find in the results:
\begin{itemize}
	\item all estimators outperform crude Monte Carlo $\hat{\alpha}_0$,
	\item the estimators which calculate $\Prob(X_i > \gamma)$ outperform those which do not,
	\item the estimators which calculate $\Prob(X_i > \gamma, X_j > \gamma)$ outperform those which only use the univariate $\Prob(X_i > \gamma)$,
	\item the importance sampling estimators improve upon their original counterparts,
	\item the second-order IS improves upon the first-order IS.
\end{itemize}

\noindent
Also noticed in the performance of the $\hat{\alpha}$ estimators:
\begin{itemize}
	\item the $\hat{\alpha}_1$ and $\hat{\alpha}_2$ estimators often degenerated (i.e.\ had zero variance) to $\alphaMax$ and $\alphaMax{-}q$ respectively,
	\item the degeneration begin for smaller $\gamma$ when the $\bfX$ had a weaker dependence structure.
\end{itemize}
Table~\ref{table:ratio-alphas} shows the degeneration of the estimators in various examples involving multivariate normal distributions.

\begin{table}
\centering
\subfloat[$\hat{\alpha}_1$ to $\alphaMax$]{
\begin{tabular}{cc|cccc}
\toprule
\multicolumn{2}{c}{Test cases} & \multicolumn{4}{c}{$\gamma$} \\
$d$                    & $\rho$    & 2    & 4    & 6    & 8    \\
\midrule
\multirow{4}{*}{3}   & -0.25   & 0.00957 & 1* & 1* & 1*    \\
                     & 0       & 0.00255 & 1* & 1* & 1*     \\
                     & 0.5     & 0.00166 & 1* & 1* & 1*      \\
                     & 0.75    & 0.005 & 0.165 & 1* & 1*      \\
\midrule
\multirow{4}{*}{4}   & -0.25   & 0.00955 & 1* & 1* & 1*     \\
                     & 0       & 0.0185 & 1* & 1* & 1*     \\
                     & 0.5     & 0.00139 & 1* & 1* & 1*      \\
                     & 0.75    & 0.00484 & 0.283 & 1* & 1*   \\
\midrule
\multicolumn{2}{c}{Average} & 0.00663 & 0.806 & 1 & 1 \\
\bottomrule
\end{tabular}}
~~~~
\subfloat[$\hat{\alpha}_2$ to $\alphaMax{-}q$]{
\begin{tabular}{cc|cccc}
\toprule
\multicolumn{2}{c}{Test cases} & \multicolumn{4}{c}{$\gamma$} \\
$d$                    & $\rho$    & 2    & 4    & 6    & 8    \\
\midrule
\multirow{4}{*}{3}   & -0.25   &  1* & 1* & 1* & 1*    \\
                     & 0       & 0.151* & 1* & 1* & 1*     \\
                     & 0.5     & 0.0764 & 1* & 1* & 1*      \\
                     & 0.75    & 0.0172 & 0.754 & 1* & 1*      \\
\midrule
\multirow{4}{*}{4}   & -0.25   & 1* & 1* & 1* & 1*     \\
                     & 0       & 0.189 & 1* & 1* & 1*     \\
                     & 0.5     & 0.0153 & 1* & 1* & 1*      \\
                     & 0.75    & 0.0175 & 0.502 & 1* & 1*  \\
\midrule
\multicolumn{2}{c}{Average} & 0.308 & 0.907 & 1 & 1 \\
\bottomrule
\end{tabular}}
\caption{Ratios of absolute relative errors for pairs of estimators. Each row corresponds to a separate distribution for $\bfX$, each being $\NormDist_d$ distributed with standard normal marginals and constant correlation $\rho$.}
\label{table:ratio-alphas}
\end{table}

The fact that the estimators degenerate is not wholly undesirable, as they degenerate to the deterministic functions $\alphaMax$ and $\alphaMax{-}q$ which are highly accurate when degeneration occurs. Obviously, for very large $\gamma$ one would not resort to Monte Carlo methods as the asymptote $\alphaMax$ would be accurate enough for most purposes; one could use the $\hat{\alpha}$ estimators until the sample variance is below some threshold, then switch to the faster deterministic estimators $\alphaMax$ and $\alphaMax{-}q$. \\

\noindent
Regarding the $\widehat{(\beta_1 \ddagger \alpha)}$ and $\widehat{(\beta_2 \ddagger \alpha)}$ estimators:
\begin{itemize}
\item their performance is roughly the same as than their $\hat{\alpha}_1^{[1]}$ and $\hat{\alpha}_2^{[2]}$ counterparts,
\item they perform better when the dependence between the variables is weak.
\end{itemize}
One must remember that the $\hat{\beta}_i$ estimators are valid for a much larger class of problems (estimating expectations, not just probabilities). Also, we would expect that the $\hat{\beta}_i$-based estimators compare favorably to the $\hat{\alpha}_i^{[i]}$ IS-based estimators when $d$ is large, as the method involves no likelihood term which can degenerate.

\section{Conclusion}

In this paper we presented new estimators for the tail probability of a union of dependent rare events.
The key idea in both estimators is that the tail probability of the such a rare event can be
well approximated by the Bonferroni approximations:
\begin{align*}
  \alpha = \Prob(A) \approx \sum_{i=1}^k (-1)^{i-1} \sum_{|I|= i} \Prob\Big(\bigcap_{i\in I}A_i\Big) \text{ for } k=1,2\,.
\end{align*}

We provided conditions which ensure $\hat{\alpha}_1$ and $\hat{\beta}_i$ have logarithmic efficiency and bounded relative error. 
The estimators were tested on the classical example of rare maxima of random vectors.
Furthermore, we note the fact that our $\hat{\beta}_i$ estimators can be applied 
to a more general setting which could make useful for a larger variety of estimation problems.

\subsection{Future work}

In this paper we did not discuss stratification strategies for $\hat{\beta}_i$ that could result in further reductions in variance. Nor did we investigate which permutations of the $A_i$ minimise the variance of $\hat{\beta}_i$. Further investigation into the use of $\hat{\beta}_i$ to estimate tail probabilities of order statistics would be of value.

\section*{Acknowledgements}
 LRN is supported by ARC grant DE130100819.


\def\cprime{$'$} \def\cprime{$'$}

\appendix

\section{Elliptical distribution asymptotics} \label{scn:elliptical_asymptotics}

\subsection{Asymptotic properties of normal distributions}

In general, for an $\bfX \sim \NormDist_d(\bfzero,\bfSigma)$, Theorem 2.6.1 of Bryc \cite{bryc2012normal} states that for all measurable $A \subset \RL^d$ the
\begin{equation} \label{bryc}
\lim_{n \to \infty} \frac1{n^2} \log \Prob(\bfX \ge nA) = - \inf_{\bfx \in A} \frac12 \bfx^\top \bfSigma^{-1} \bfx.
\end{equation}
The asymptotic properties of elliptical distributions also relate to this quadratic programming problem, which Hashorva \cite{hashorva2005asymptotics,hashorva2007asymptotic} denotes as
\begin{equation} \label{quad_prog}
\mathcal{P}(\bfSigma^{-1}, \bft) := \text{minimise } \bfx^\top \bfSigma^{-1} \bfx \text{ under the linear constraint } \bfx \ge \bft.
\end{equation}
The program $\mathcal{P}(\bfSigma^{-1}, \bft)$ is usually minimised at the boundary $\bft$, and hence the asymptotic form \eqref{bryc} is very simple. This occurs when
$\bfSigma^{-1} \bft > \bfzero$ (componentwise), a condition often called the \emph{Savage condition} after Richard Savage \cite{savage1962mills}. For the cases when the Savage condition fails, the asymptotics change as some components of $\bfX$ become irrelevant in the limit. Figure~\ref{fig:savage_condition} graphically shows some contours of $\bfx^\top \bfSigma^{-1} \bfx$ for some $\bfSigma$ which do and do not satisfy the Savage condition.

\begin{figure}[H]
\centering
\includegraphics[width=0.3\textwidth]{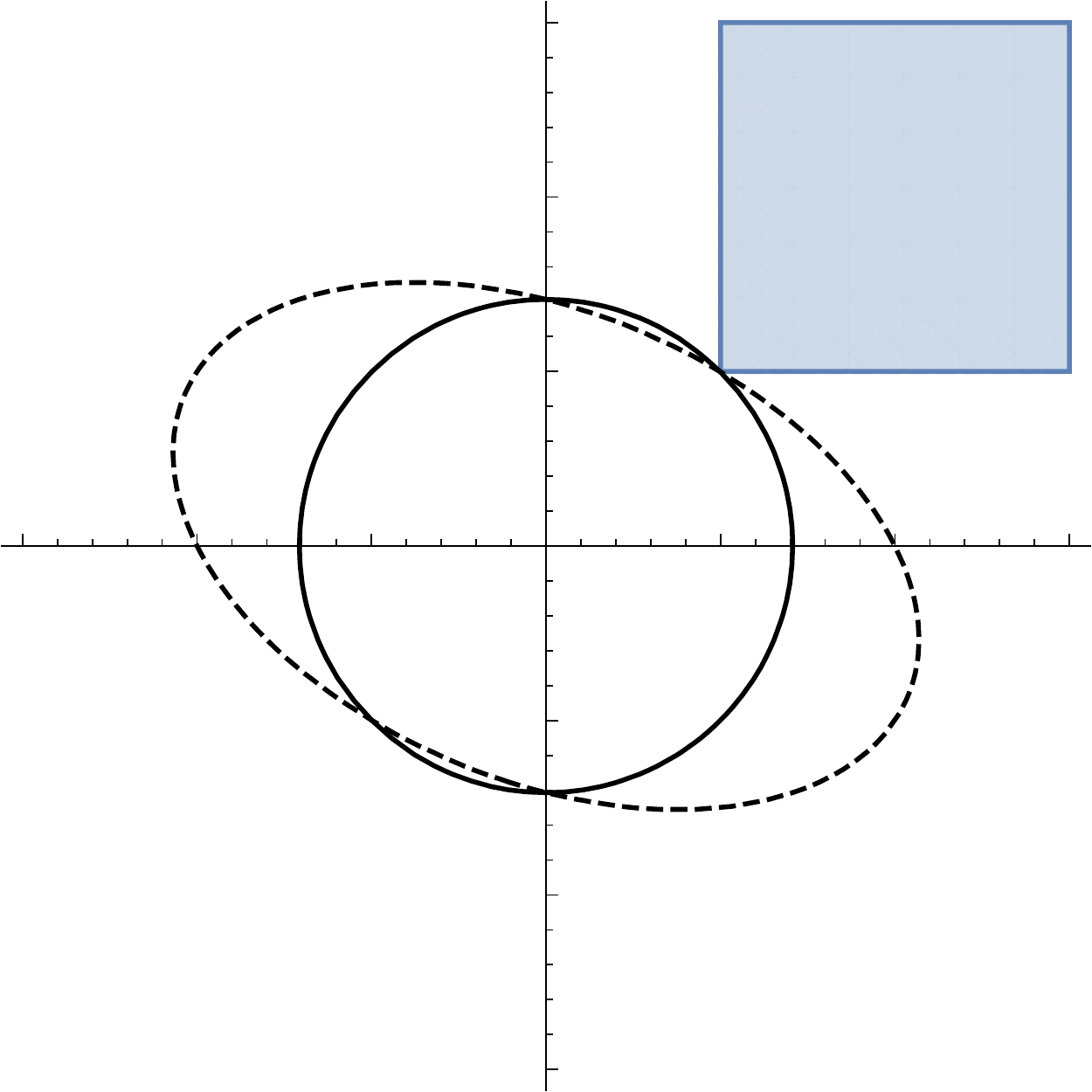}~~
\includegraphics[width=0.3\textwidth]{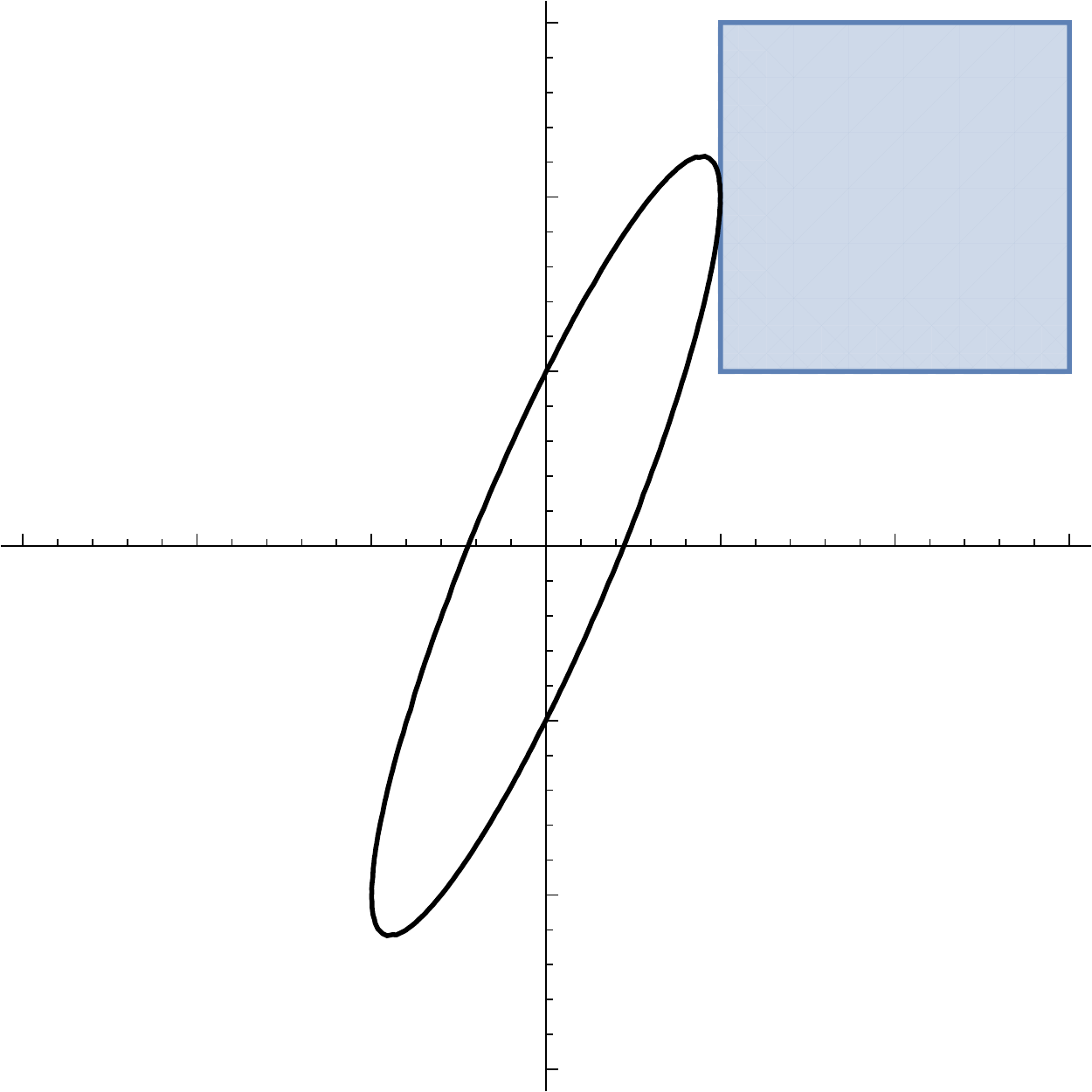}~~
\includegraphics[width=0.3\textwidth]{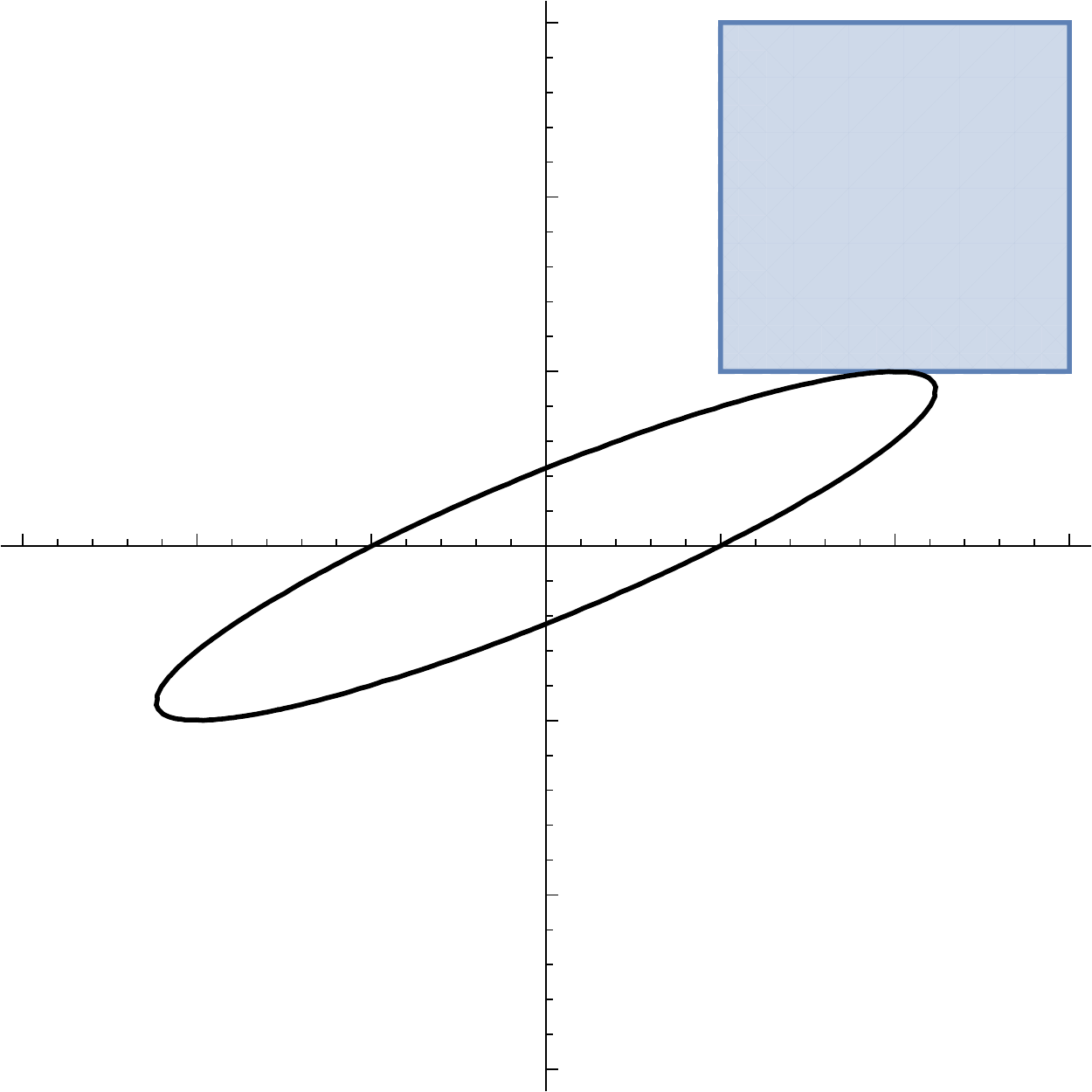}
\caption{Contours of $\bfx^\top \bfSigma^{-1} \bfx$ for example $\bfSigma$ which: (a) satisfy the Savage condition (i.e., $\bfSigma^{-1} \bfone > \bfzero$), and (b)--(c) do not satisfy the condition. The covariance matrices, in \textsc{Matlab} notation, are: (a) $\bfSigma = \bfI$ and  $\bfSigma =[2, -1/2; -1/2, 1]$, (b) $\bfSigma = [1, 2; 2, 5]$, and (c) $\bfSigma = [5, 2; 2, 1]$.}
\label{fig:savage_condition}
\end{figure}

\subsection{Asymptotic properties of type I elliptical distributions}

Take $\bfX \sim \EllDist(\bfmu, \bfSigma, F)$ where the radial distribution $F \in$  MDA(Gumbel) has support $(0,x_F)$, for some $x_F \in \overline{\RL}$, and where $\{\sigma_1, \dots, \sigma_d\}$ are in decreasing order. 
The univariate and bivariate asymptotics, $\Prob(X_i > \gamma)$ and $\Prob(X_i > \gamma, X_j > \gamma)$, can be written in terms of the scaling function $w(\gamma)$ and of $\Ftail((\gamma-\mu)/\kappa)$ for some particular $\mu$ and $\kappa$.
Theorem 12.3.1 of Berman \cite{berman1992sojourns} gives the univariate case,
\begin{equation} \label{berman}
 \Prob(X_i > \gamma) =  (1+\oh(1)) \frac{ \Ftail(\upsilon_i(\gamma)) }
 {\sqrt{2\pi \upsilon_i(\gamma) w(\upsilon_i(\gamma)) }} \quad \text{ as } \gamma \to x_F
\end{equation}
where $\upsilon_i(\gamma) = (\gamma - \mu_i)/\sigma_i$. The bivariate case, i.e.\ $\Prob(X_i > \gamma, X_j > \gamma)$, relies on the following constants. Define $a_{ij} := \sigma_j/\sigma_i$. If $\rho_{ij} \ge a_{ij}$ then define
 \[ \mu_{ij}:= \mu_j \quad \text{ and } \quad \kappa_{ij}:= \sigma_j \]
 otherwise for $\rho_{ij} < a_{ij}$
 \[ \mu_{ij}:= \frac{\mu_i-a_{ij}\rho_{ij}(\mu_1+\mu_2)+a^2\mu_j}{\alpha_{ij}(1-\rho^2_{ij})} \quad \text{ and } \quad \kappa_{ij}:= \frac{\sigma_i^2\sigma_j^2(1-\rho^2_{ij})}{ {\sigma_i^2-2\rho_{ij}\sigma_i\sigma_j + \sigma_{j}^2}}. \]

\begin{Theorem} \label{thm:type_1_asymptotics}
 Let $(X_i, X_j)$ be a pair from a type I elliptical random vector $\bfX \sim E(\bfmu, \bfSigma, F)$ 
 and consider $\gamma \nearrow x_F$.
Then with $\upsilon_{ij}(\gamma)=(\gamma-\mu_{ij})/\kappa_{ij}+c_{ij}(\gamma)$ for some $c_{ij}(\gamma) \in \oh(1)$,
 \begin{align*}
   \Prob(X_i > \gamma, X_j > \gamma)
   &= \Ftail(\upsilon_{ij}(\gamma)) \times \begin{cases}
   		 \Big( 2\pi \upsilon_{ij}(\gamma) w(\upsilon_{ij}(\gamma)) \Big)^{-1/2}(1 + \oh(1)),
   		 & \text{ if } \rho_{ij} > a_{ij}, \\
         \Big( 2\pi \upsilon_{ij}(\gamma) w(\upsilon_{ij}(\gamma)) \Big)^{-1}(C_{a,\rho} + \oh(1)),
         & \text{ if } \rho_{ij} < a_{ij},
   	\end{cases}
 \end{align*}
 for a $C_{a,\rho} \in \RL_+$.
Furthermore, if either $\mu_i\ge\mu_j$ or $\lim_{\gamma\to x_F}w(\gamma)/\gamma<\infty$, then there exists a $C_\rho \in \RL_+$ such that
 \begin{align*}
   \Prob(X_i > \gamma, X_j > \gamma)
   &= \Ftail(\upsilon_{ij}(\gamma)) \Big( 2\pi \upsilon_{ij}(\gamma) w(\upsilon_{ij}(\gamma)) \Big)^{-1/2}(C_\rho + \oh(1)),
   		 & \text{ if } \rho_{ij} = a_{ij}.
 \end{align*}
\end{Theorem}

\begin{proof}

Use Theorem 2 of Hashorva \cite{hashorva2007asymptotic}.
First we consider the case $a_{ij}<\rho_{ij}$.  In such a case it holds that
\begin{align*}
 \lim_{\gamma\to x_F} \sqrt{\frac{w(\upsilon_j(\gamma))}{\upsilon_j(\gamma)}}\left(\upsilon_i(\gamma)-\rho_{ij}\upsilon_j(\gamma)\right)
 &=\lim_{\gamma\to x_F} \sqrt{{w(\upsilon_j(\gamma))\upsilon_j(\gamma)}}\left(\frac{\upsilon_i(\gamma)}{\upsilon_j(\gamma)}-\rho_{ij}\right)\\
 &=\lim_{\gamma\to x_F} \sqrt{{w(\upsilon_j(\gamma))\upsilon_j(\gamma)}}\left(a_{ij}-\rho_{ij}\right)=-\infty.
\end{align*}
Hence, the hypotheses of Case i) of Theorem 2 of Hashorva \cite{hashorva2007asymptotic} hold and the first result follows.
In the case where $a_{ij}=\rho_{ij}$ then
\[
 \lim_{\gamma\to x_F} \sqrt{\frac{w(\upsilon_j(\gamma))}{\upsilon_j(\gamma)}}\left(\upsilon_i(\gamma)-\rho_{ij}\upsilon_j(\gamma)\right)
 =\lim_{\gamma\to x_F} \sqrt{\frac{w(\upsilon_j(\gamma))}{\upsilon_j(\gamma)}}\frac{(\mu_j-\mu_{i})}{\sigma_i}.
\]
The last limit remains bounded from above if either $\mu_i>\mu_j$ or $\lim_{\gamma\to\infty}w(\gamma)/\gamma<\infty$.
For the case $a_{ij}>\rho_{ij}$ we define
$a_{ij}(\gamma) := \upsilon_i(\gamma) / \upsilon_j(\gamma)$ so $\lim_{\gamma \to \infty} a_{ij}(\gamma)=a_{ij}$.

We let
 \[
   \tau_{ij}(\gamma) =
   \sqrt{ \frac{1-2a_{ij}(\gamma)\rho_{ij} + a_{ij}^2(\gamma)}{1-\rho^2_{ij}}},\quad
  \tau_{ij}:=\lim_{\gamma\to\infty}\tau_{ij}(\gamma)
    =\sqrt{ \frac{1-2a_{ij}\rho_{ij} + a_{ij}^2}{1-\rho^2_{ij}}}.
\]
The results follows by noting that
\[
 \upsilon_j(\gamma)\tau_{ij}(\gamma)= \upsilon_{ij}(\gamma),\qquad
 \upsilon_{ij}(\gamma)= \frac{\gamma-\mu_{ij}}{\tau_{ij}}+\oh(1).
\]
\qed
\end{proof}

\end{document}